\documentclass[a4paper,11pt]{article}
\usepackage[utf8]{inputenc}
\usepackage[T1]{fontenc}
\usepackage{amsmath}
\usepackage{amssymb}
\usepackage{graphicx}
\usepackage{xcolor}
\usepackage{mathtools}
\usepackage{hyperref}

\renewcommand{\leq}{\leqslant}
\renewcommand{\geq}{\geqslant}

\newcommand{\Norm}[1]{\left\|#1\right\|}
\newcommand{\NormInfinity}[1]{\left\|#1\right\|_{\infty}}
\newcommand{\hNorm}[1]{\left\|#1\right\|_{\mathcal{H}}}
\newcommand{\BlackBox}{\rule{1.5ex}{1.5ex}}  
\newcommand{\qed}{\hfill\BlackBox\\[2mm]}

\newcommand{\obspace}{\Xi}
\newcommand{\parspace}{\mathbb{S}}
\newcommand{\risk}{\mathcal{L}}
\newcommand{\cfun}{\gamma}
\DeclarePairedDelimiter\floor{\lfloor}{\rfloor}

\newcommand{\cM}{\mathcal{M}}
\newcommand{\cT}{\mathcal{T}}
\newcommand{\cX}{\mathcal{X}}
\newcommand{\cY}{\mathcal{Y}}
\newcommand{\bE}{\mathbb{E}}
\newcommand{\un}{\mathbb{I}}
\newcommand{\bayes}{s}
\newcommand{\loss}[1]{\ell(s,#1)}
\newcommand{\lossb}[1]{\ell\bigl(s,#1 \bigr)}
\newcommand{\xloss}[1]{\ell_X(#1)}

\newcommand{\cpred}{g}
\newcommand{\ctrain}{c}
\newcommand{\cvxct}{\rho}
\newcommand{\slope}{\nu}
\newcommand{\learnrule}{\mathcal{A}}
\newcommand{\ind}{j}

\newcommand{\ERM}[1]{\widehat{f}_{#1}}
\newcommand{\ERMho}[1]{\widehat{f}^{\,\mathrm{ho}}_{#1}}
\newcommand{\ERMag}[1]{\widehat{f}^{\,\mathrm{ag}}_{#1}}
\newcommand{\ERMcv}[1]{\widehat{f}^{\,\mathrm{cv}}_{#1}}
\newcommand{\ERMmv}[1]{\widehat{f}^{\,\mathrm{mv}}_{#1}}
\newcommand{\HO}[2]{\text{HO}_{#1}\left(#2\right)}
\newcommand{\numberthis}{\addtocounter{equation}{1}\tag{\theequation}}
\newenvironment{proof}{\par\noindent{\bf Proof\ }}{\qed}
\DeclareMathOperator*{\argmin}{argmin}
\DeclareMathOperator*{\argmax}{argmax}

\newtheorem{lemma}{Lemma}[section]
\newtheorem{Theorem}[lemma]{Theorem}
\newtheorem{corollary}[lemma]{Corollary}
\newtheorem{proposition}[lemma]{Proposition}
\newtheorem{Definition}[lemma]{Definition}
\newtheorem{claim}{Claim}[lemma]

\newtheorem{example}{Example}[section]
\newtheorem{Remark}{Remark}[section]
\title{Aggregated Hold-Out}
\author{Guillaume Maillard \\ \href{mailto:guillaume.maillard@u-psud.fr}{guillaume.maillard@u-psud.fr} \and Sylvain Arlot \\ \href{mailto:sylvain.arlot@u-psud.fr}{sylvain.arlot@u-psud.fr} \and Matthieu Lerasle \\ \href{mailto:matthieu.lerasle@u-psud.fr}{matthieu.lerasle@u-psud.fr}
\vspace*{.2cm}
\\ 
Laboratoire de Math\'ematiques d'Orsay, Univ. Paris-Sud, \\ 
CNRS, Universit\'e Paris-Saclay, 91405 Orsay, France. 
\vspace*{.2cm}
\\ 
Inria Saclay - Ile-de-France, B\^at. Turing, \\ 
Campus de l'Ecole Polytechnique, 91120 Palaiseau, France
}

\begin{document}

\maketitle

\begin{abstract}
Aggregated hold-out (Agghoo) is a method which averages learning rules selected by hold-out 
(that is, cross-validation with a single split). 
We provide the first theoretical guarantees on Agghoo, 
ensuring that it can be used safely: 
Agghoo performs at worst like the hold-out when the risk is convex. 
The same holds true in classification with the 0--1 risk, with an additional constant factor. 
For the hold-out, oracle inequalities are known for bounded losses, as in binary classification. 
We show that similar results can be proved, under appropriate assumptions, for other risk-minimization problems. 
In particular, we obtain an oracle inequality for regularized kernel regression with a Lipschitz loss, 
without requiring that the $Y$ variable or the regressors be bounded.
Numerical experiments show that aggregation brings a significant improvement over the hold-out and 
that Agghoo is competitive with cross-validation. 
\end{abstract}


\section{Introduction}

%
The problem of choosing from data among a family of learning rules is central to machine learning. 
There is typically a variety of rules which can be applied to a given problem 
---for instance, support vector machines, neural networks or random forests. 
Moreover, most
machine learning rules depend on hyperparameters which have a strong impact on the final performance of the algorithm. 
For instance, $k$-nearest-neighbors rules  \cite{Bia_Dev:2015} 
depend on the number $k$ of neighbors. 
%
%
A second example, among many others, is given by regularized empirical risk minimization rules, 
such as support vector machines \cite{Ste_Chr:2008} or the Lasso \cite{Tib:1996,Buh_vdG:2011}, 
which all depend on some regularization parameter. 
A related problem is model selection \cite{Bur_And:2002,Mas:2003:St-Flour}, where one has to choose 
among a family of candidate models. 

%

%
In supervised learning, 
cross-validation (CV) is a general, efficient and classical answer to the problem 
of selecting a learning rule \cite{Arl_Cel:2010:surveyCV}. 
It relies on the idea of splitting data into a training sample 
---used for training a predictor with each rule in competition--- 
and a validation sample 
---used for assessing the performance of each predictor. 
This leads to an estimator of the risk 
---the hold-out estimator when data are split once, 
the CV estimator when an average is taken over several data splits---, 
which can be minimized for selecting among a family of competing rules.

%
A completely different strategy, called aggregation, 
is to \emph{combine} the predictors obtained with all candidates 
\cite{Nem:2000,Yan:2001,Tsy:2004}.  
Aggregation is the key step of ensemble methods \cite{Die:2000}, 
among which we can mention bagging \cite{Bre:1996a}, 
AdaBoost \cite{Fre_Sch:1997} 
and random forests \cite{Bre:2001,Bia_Sco:2016:TEST}. 
A major interest of aggregation is that it builds a learning rule 
that may not belong to the family of rules in competition. 
Therefore, it sometimes has a smaller risk than 
the best of all rules \cite[Table~1]{Salmon_Dal:2011}. In contrast, cross-validation, which selects only one candidate, cannot outperform the best rule in the family. 
%

\paragraph{Aggregated hold-out (Agghoo)}
%
%
This paper studies a procedure mixing cross-validation and aggregation ideas, that we
call \emph{aggregated hold-out} (Agghoo). 
Data are split several times; 
for each split, the hold-out selects one predictor; 
then, the predictors obtained with the different splits are aggregated. 
A formal definition is provided in Section~\ref{sec.definitions}. 
This procedure is as general as cross-validation 
and it has roughly the same computational cost (see Section~\ref{sec.comp}). 
Agghoo is already popular among practicioners, 
and has appeared in the neuro-imaging literature \cite{HoyosIdrobo2015,Varoquaux2017} 
under the name ``CV + averaging''. 
%
%
Yet, to the best of our knowledge, existing experimental studies do not give any indication on how to choose Agghoo's parameters.
No general mathematical definition has been provided, so it is unclear how to generalize Agghoo beyond a given article's setting. 
Theoretical guarantees on Agghoo have not been established yet, 
to the best of our knowledge. 
The closest results we found 
study other procedures, called ACV \cite{Jun_Hu:2015}, EKCV  \cite{Jun:2016}, 
or ``bagged cross-validation'' \cite{Hall2009}, 
and they do not prove oracle inequalities. 
We explain in Section~\ref{sec_def_agghoo} 
why Agghoo should be preferred to these procedures in the general prediction setting.

%
Because of the aggregation step, Agghoo is an ensemble method, and like bagging, it combines resampling with aggregation. The application of bagging to the hold-out was first suggested by Breiman \cite{Bre:1996a} as a way to combine pruning and bagging of CART trees. 
The combination of bagging and cross-validation has been studied numerically by \cite{Petersen2007}. A major difference with Agghoo is that the training and validation samples are not independent with bagging, which uses sampling \emph{with replacement}.
If the bootstrap is replaced by subsampling, bagging becomes subagging \cite{Buh_Yu:2002}, 
and its combination with cross-validation yields a procedure much closer to Agghoo, but still different, 
see Section~\ref{sec_def_agghoo}. 
Overall, previous results on bagging or subagging do not apply to Agghoo; new developments are required. 

\medbreak

\paragraph{Contributions} 
In this article, Agghoo's performance is studied both theoretically and experimentally.
We consider Agghoo from a prediction point of view. Performance is measured by a risk functional. 
On the theoretical side, the aim is to show that the risk of Agghoo's final predictor 
is as low as the risk of the optimal rule among the given collection. This is known as an oracle inequality. By a convexity argument, Agghoo always improves on the hold-out, provided that the risk is convex. Hence, Agghoo can safely replace the hold-out in any application where this hypothesis holds true. Another consequence is that oracle inequalities for Agghoo can be deduced from oracle inequalities for the hold-out.

This kind of result on the hold-out has already appeared in the literature: for example, Massart \cite[Corollary 8.8]{Mas:2003:St-Flour} proves a general theorem under an abstract noise assumption; more explicit results have been obtained in specific settings such as least-squares regression \cite[Theorem 7.1]{Gyrfi2002} or maximum-likelihood density estimation \cite[Theorem 8.9]{Mas:2003:St-Flour}. A review on cross-validation ---which includes the hold-out--- can be found in \cite{Arl_Cel:2010:surveyCV}. 

Most existing theoretical guarantees on the hold-out have a limitation: they assume that the loss function is uniformly bounded. In regression, the variable $Y$ and the regressors are also usually assumed to be bounded, which excludes some standard least-squares estimators. 
Even when the boundedness assumption holds true, constants arising from general bounds may be of the wrong order of magnitude, leading to vacuous results. 
By replacing uniform supremum bounds by local ones, we are able to relax these hypotheses in a general setting (Theorem~\ref{agcv_mean}). 
This enables us to prove an oracle inequality for the hold-out and Agghoo in regularized kernel regression  with a general Lipschitz loss (Theorem~\ref{rkhs_thm}). 
This oracle inequality allows for instance to recover state-of-the-art convergence rates in median regression without knowing the regularity of the regression function (adaptivity), both in the general case and, for small enough regularity, also in the specific setting of \cite{eberts2013}.
To illustrate the implications of Theorem~\ref{rkhs_thm}, we also apply it to $\varepsilon$-regression (Corollary~\ref{eps_reg}).
To the best of our knowledge, all these oracle inequalities are new, even for the hold-out.

A limitation of Agghoo is that it does not cover settings where averaging does not make sense,
such as classification. In classification with the 0--1 loss, the natural way to aggregate classifiers
is to take a majority vote among them. This yields a procedure which we call Majhoo.
Using existing theory for the hold-out in classification, we prove that Majhoo satisfies a general, 
margin-adaptive oracle inequality (Theorem~\ref{thm_classif}) under Tsybakov's margin assumption \cite{Mam_Tsy:1999}.

All our oracle inequalities are valid for any size of the aggregation ensemble.
Qualitatively, since bagging and subagging are well-known for their stabilizing effects \cite{Bre:1996a,Buh_Yu:2002},  
we can expect Agghoo to behave similarly.  
In particular, large ensembles should improve much the prediction performance of CV 
when the hold-out selected predictor is unstable.

For further insights into Agghoo and Majhoo, we conduct in Section~\ref{sec_simus} 
a numerical study on simulated datasets. 
Its results confirm our intuition: in all settings considered, Agghoo and Majhoo actually perform much better than the hold-out, 
and even better than CV, provided their parameters are well-chosen. 
When choosing the number of neighbors for $k$-nearest neighbors, 
the prediction performance of Majhoo is much better than the one of CV, 
which illustrates the strong interest of using Agghoo/Majhoo when 
learning rules are ``unstable''.
In support vector regression, Agghoo can even perform better than the oracle, an improvement made possible by aggregation, that cannot be matched by any hyperparameter selection rule.
Based upon our experiments, we also give in Section~\ref{sec_simus} 
some guidelines for choosing Agghoo's parameters: 
the training set size and the number of data splits.
\medbreak
The remaining of the article is structured as follows.
In Section $2$, we introduce the general statistical setting. In Section $3$, we give a formal definition of Agghoo. In Section $4$, we state the main theoretical results. In Section $5$, we present our numerical experiments and discuss the results. Finally, in Section $6$, we draw some qualitative conclusions about Agghoo. 
The proofs are postponed to the Appendix.

\section{Setting and Definitions}
We consider a general statistical learning setting, following the book by Massart \cite{Mas:2003:St-Flour}.
\subsection{Risk minimization} \label{risk_minim}
The goal is to minimize over a set $\parspace$ a risk functional $\risk: \parspace \rightarrow \mathbb{R} \cup \{ + \infty \}$.
The set $\parspace$ may be infinite dimensional for non-parametric problems.
Assume that $\risk$ attains its minimum over $\parspace$ at a point $\bayes$, called
a Bayes element. 
Then the \emph{excess risk} of any $t \in \mathbb{S}$ is the nonnegative quantity
\[ \loss{t} = \risk(t) - \risk(s) \enspace . \]
Suppose that the risk can be written as an expectation over an unknown probability distribution:
\[ \risk(t) = \mathbb{E}\bigl[ \cfun(t,\xi) \bigr] \enspace,\]
for a \emph{contrast function} $\cfun: \parspace \times \obspace \rightarrow \mathbb{R}$ and a random variable $\xi$ with values in some set $\obspace$
and unknown distribution $P$,
such that 
\[ \forall t \in \parspace, \qquad \widetilde{\xi} \in \Xi \mapsto \gamma(t,\widetilde{\xi}) \text{ is } P\text{-measurable}\enspace. \]
The statistical learning problem is to use data $D_n = \{ \xi_1,...,\xi_n \}$, where $\xi_1,...,\xi_n$ are independent and identically distributed (i.i.d.), with common distribution $P$,
to find an approximate minimizer for $\risk$.
The quality of this approximation is measured 
by the excess risk.

\subsection{Examples} 
\emph{Supervised learning} aims at predicting a quantity of interest $Y \in \mathcal{Y}$ using explanatory variables $X \in \cX$. 
The statistician observes pairs $(X_1,Y_1),\ldots (X_n,Y_n)$, so that $\obspace = \cX \times \cY$, and seeks 
a predictor in $\parspace = \{ t: \mathcal{X} \rightarrow \mathcal{Y}: \text{t measurable} \}$. 
The contrast function is defined by $\cfun(t,(x,y)) = g(t(x),y)$ for some \emph{loss function} $g: \cY \times \cY \rightarrow \mathbb{R}$.
Here, $g(y',y)$ measures the loss incurred by predicting $y'$ instead of the observed value $y$.
 Two classical supervised learning problems are classification and regression, which we detail below.

%

\begin{example}[Classification] \label{classif_pbm}
 In classification $Y$ belongs to a finite set of labels $\mathcal{Y} = \{0,\ldots,M \}$. We wish to correctly label any new data point $X$, and the risk is the probability of error\textup{:} 
 \[ \forall t \in \parspace, \qquad \risk(t) = \mathbb{P} \bigl( t(X) \neq Y \bigr) \enspace,\]
 which corresponds to the loss function $ g(y',y) = \mathbb{I}\{ y' \neq y \}$. Classification with convex losses \textup{(}such as the hinge loss or logistic loss\textup{)} can also be described using the formalism of Section~\ref{risk_minim}. 
\end{example}


\begin{example}[Regression] \label{def_reg}

 In regression we wish to predict a continuous variable $Y \in \mathcal{Y} = \mathbb{R}^d$.
 The error made by predicting $y'$ instead of $y$ is measured by the loss function defined by $g(y',y) = \phi(\Norm{y' - y})$
 where $\phi: \mathbb{R}_+ \rightarrow \mathbb{R}_+$ is nondecreasing and convex.
 Some typical choices are $\phi(x) = x^2$ \textup{(}least squares\textup{)}, $\phi(x) = x$ \textup{(}median regression\textup{)} or
 $\phi(x) = \left( |x| - \varepsilon \right)_+$ \textup{(}Vapnik's $\varepsilon$-insensitive loss, leading to $\varepsilon$-regression\textup{)}.
The risk is given by
\[ \risk(t) = \mathbb{E}\Bigl[\phi\bigl(\Norm{Y - t(X)}\bigr) \Bigr] \enspace.\]
If $\phi$ is strictly convex, the minimizer of $\risk$ over $\parspace$ is a unique function, up to modification on 
a set of probability $0$ under the distribution of $X$.
\end{example}

In some applications, such as robust regression, it is of interest to define $\bayes$ and $\loss{t}$ even when $\phi(\Norm{Y}) \notin L^1$.
This is possible for Lipschitz contrasts, by the following remark.

\begin{Remark}
 When $\phi$ is convex and increasing \textup{(}as in Example~\ref{def_reg}\textup{)}, and also Lipschitz-continuous, it is always possible to define
 \[ \bayes: x \mapsto \argmin_{u \in \mathbb{R}} \mathbb{E} 
 \bigl[ \phi(\Norm{Y-u}) - \phi(\Norm{Y}) \, \big\vert \, X = x \bigr] 
 \enspace . \]
 When $\bayes \in L^1(X)$, it is a Bayes element for the loss function $g(y',y) = \phi(\Norm{y' - y}) - \phi(\Norm{y})$.
Whenever $\phi(\Norm{Y}) \in L^1$, this loss yields the same Bayes element and excess risk as in Example 2.2.
\end{Remark}
 This small adjustment to the general definition allows to consider Example~\ref{def_reg} when $\phi(\Norm{Y - \bayes(X)})$ is not integrable, 
 for example when $Y = \bayes(X) + \eta$, where $\eta$ is independent from $X$ and follows a multivariate Cauchy distribution with location parameter $0$.
 \medbreak
 Some density estimation problems, such as maximum likelihood or least-squares density estimation, also fit the formalism of Section~\ref{risk_minim}, see \cite{Mas:2003:St-Flour}. 
 
\subsection{Learning rules and estimator ensembles} 
Statistical procedures use data to compute an element of $\parspace$ which approximately minimizes $\risk$.
Since Agghoo uses subsampling, we require learning rules to accept as input datasets of any size.
Therefore, we define a learning rule to be a function which maps any dataset to an element of $\parspace$.

\begin{Definition}
A \emph{dataset} $D_n$ of length $n$ is a finite i.i.d sequence $(\xi_i)_{1 \leq i \leq n}$ of $\obspace$-valued random variables with common distribution~$P$.

  A \emph{learning rule} $\learnrule$ is a measurable function\footnote{For any $n$,
  \begin{equation*}
   \begin{cases}
    \obspace^n \times \obspace &\rightarrow \mathbb{R} \\
    (\xi_{1:n}, \xi) &\mapsto \cfun(\learnrule(\xi_{1:n}),\xi)
   \end{cases}
  \end{equation*}
is assumed to be measurable (with respect to the product $\sigma$-algebra on $\obspace^{n+1}$).}
  \[\learnrule: \bigcup_{n = 1}^{\infty} \obspace^n \rightarrow \parspace \enspace . \]
\end{Definition}

In the risk minimization setting,
$\learnrule$ should be chosen so as to minimize $\risk(\learnrule(D_n))$.
\medbreak
A generic situation is when a family $(\learnrule_m)_{m \in \cM}$ of learning rules is given, 
so that we have to select one of them (estimator selection), or to combine their outputs (estimator aggregation). 
For instance, when $\cX$ is a metric space, we can consider the family 
$(\learnrule_k^{\mathrm{NN}})_{k \geq 1}$ of nearest-neighbors classifiers 
---where $k$ is the number of neighbors---, 
or, for a given kernel on $\cX$, the family $(\learnrule_{\lambda}^{\mathrm{SVM}})_{\lambda \in [0,+\infty)}$ 
of support vector machine classifiers 
---where $\lambda$ is the regularization parameter. 
Not all rules in such families perform well on a given dataset. Bad rules should be avoided when selecting the hyperparameter, or be given small weights if the outputs are combined in a weighted average. 
 This requires a data-adaptive procedure, as the right choice of rule in general depends on the unknown distribution $P$.
\medbreak
Aggregation and parameter selection methods aim to resolve this problem,
as described in the next section.

%
\section{Cross-Validation and Aggregated Hold-Out (Agghoo)} \label{sec.definitions}
This section recalls the definition of cross-validation for estimator selection,
and introduces a new procedure called aggregated hold-out (Agghoo). For more details
and references on cross-validation, we refer the reader to the survey by Arlot and Celisse \cite{Arl_Cel:2010:surveyCV}. 
%
\subsection{Background: cross-validation}
Cross-validation uses subsampling and the empirical risk. We introduce first some notation. 
\begin{Definition}[Empirical risk]
For any dataset $D_n = (\xi_i)_{1 \leq i \leq n}$ and any $t \in \parspace$, the empirical risk
of $t$ over $D_n$ is defined by
 \[ P_n \gamma (t,\cdot) = \frac{1}{n} \sum_{i = 1}^n \gamma(t,\xi_i)\enspace. \]
For any nonempty subset $T \subset \{ 1,\ldots,n \}$, let also
\[ D_n^T = (\xi_i)_{i \in T} \]
be the subsample of $D_n$ indexed by $T$, and define the associated empirical risk by
 \[ \forall t \in \parspace, \qquad P_n^T \gamma(t,\cdot) = \frac{1}{|T|} \sum_{i \in T} \gamma(t,\xi_i) \enspace.\]
\end{Definition}
The most classical estimator selection procedure is to \emph{hold out} some data to calculate the empirical risk of each
estimator, and then select the estimator with the lowest empirical risk. This ensures that the data used to evaluate the risk are independent
from the training data used to compute the learning rules.
\begin{Definition}[Hold-out] \label{def_ho}
For any dataset $D_n$ and any subset $T \subset \{ 1,\ldots,n \}$,
 the associated hold-out risk estimator of a learning rule $\learnrule$ is defined by
 \[ \HO{T}{\learnrule, D_n} = P_n^{T^c} \gamma \left( \learnrule(D_n^T),\cdot \right) \enspace.\]
 Given a collection of learning rules $(\learnrule_m)_{m \in \cM}$, the hold-out procedure selects 
 \[ \widehat{m}_{T}^{ho}(D_n) \in \argmin_{m \in \cM} \HO{T}{\learnrule_m,D_n} \enspace, \]
 measurably with respect to $D_n$.
 The overall learning rule is then given by
 \[ \ERMho{T} \bigl( (\learnrule_m)_{m \in \cM}, D_n \bigr) = \learnrule_{\widehat{m}_{T}^{ho}(D_n)}(D_n^T) \enspace.\]
\end{Definition}
 
Hold-out depends on the arbitrary choice of a training set $T$, 
and is known to be quite unstable, 
despite its good theoretical properties \cite[Section~8.5.1]{Mas:2003:St-Flour}. 
%
Therefore, practicioners often prefer to use cross-validation instead, 
which considers several training sets.

 
\begin{Definition}[Cross-validation] \label{def_cv}
Let $D_n$ denote a dataset. Let $\cT$ denote a collection of nonempty subsets of $\{ 1, \ldots, n \}$.
The associated cross-validation risk estimator of a learning rule $\learnrule$ is defined by
\[ CV_{\cT}(\learnrule, D_n) = \frac{1}{\lvert \cT \rvert} \sum_{T \in \cT} \HO{T}{\learnrule,D_n}. \]
The cross-validation procedure then selects
\[ \widehat{m}_{\cT}^{cv}(D_n) \in \argmin_{m \in \cM} CV_{\cT}(\learnrule_m,D_n)\enspace.  \]
The final predictor obtained through this procedure is
\[ \ERMcv{\cT}\bigl( (\learnrule_m)_{m \in \cM}, D_n \bigr)  = \learnrule_{\widehat{m}_{\cT}^{cv}(D_n)}(D_n) \enspace.\]
\end{Definition}

Depending on how $\cT$ is chosen, this can lead to 
leave-one-out, leave-$p$-out, $V$-fold cross-validation 
or Monte-Carlo cross-validation, among others \cite{Arl_Cel:2010:surveyCV}. 
In the following, we omit some of the arguments $\learnrule, D_n$ which
appear in Definitions \ref{def_ho} and~\ref{def_cv}, when they are clear from context. For example, we often write $\HO{T}{\learnrule}, \widehat{m}_T^{ho}, \ERMho{T}$ instead of $\HO{T}{\learnrule,D_n}, \widehat{m}_T^{ho}(D_n), \ERMho{T} \bigl( (\learnrule_m)_{m \in \cM}, D_n \bigr)$ (respectively).

\subsection{Aggregated hold-out (Agghoo) estimators} \label{sec_def_agghoo}
In this paper, we study another way to improve on the stability of hold-out selection, 
by \emph{aggregating} the predictors $\ERMho{T}$ obtained by the hold-out procedure applied repeatedly with different training sets $T \in \cT$. 
When $\parspace$ is convex (e.g., regression), \emph{aggregated hold-out} 
(Agghoo) consists in averaging them.

\begin{Definition}[Agghoo] \label{agcv_def} 
 Assume that $\parspace$ is a convex set.
 Let $(\learnrule_m)_{m \in \cM}$ denote a collection of learning rules,
 $D_n$ a dataset, and $\cT$ a collection of subsets of $\{1,\ldots,n\}$.
Using the notation of Definition~\ref{def_ho}, the associated Agghoo estimator is defined by
\[ \ERMag{\cT} \bigl( (\learnrule_m)_{m \in \cM}, D_n \bigr) = \frac{1}{\lvert \cT \rvert} \sum_{T \in \cT} \ERMho{T} \bigl( (\learnrule_m)_{m \in \cM}, D_n \bigr) \enspace. \]
\end{Definition}

In the classification framework, as seen in Example~\ref{classif_pbm},
$\parspace = \{ f : \mathcal{X} \rightarrow \{ 0,\ldots,M \} \}$
which is not convex. 
However, there is still a natural way to aggregate several classifiers,
by taking a majority vote.
\begin{Definition}[Majhoo] \label{agcv_def_classif} 
 Let $\cY = \{ 0,\ldots,M \} $ be the set of labels.
 Given a collection of learning rules $(\learnrule_m)_{m \in \cM}$,
 a dataset $D_n$ and a collection $\cT$ of subsets of $\{1,\ldots,n\}$, the majority hold-out \textup{(}Majhoo\textup{)}  classifier is 
 any measurable $\ERMmv{\cT}\bigl( (\learnrule_m)_{m \in \cM}, D_n \bigr): \cX \rightarrow \cY$ such that, using the notation $\ERMho{T}$ introduced in Definition~\ref{def_ho}, for all $x \in \cX$,
\[\ERMmv{\cT}\bigl( (\learnrule_m)_{m \in \cM}, D_n \bigr)(x) \in \argmax_{j \in \cY} \Bigl\lvert \Bigl\{ T \in \cT \,\big\vert\, \ERMho{T}\bigl( (\learnrule_m)_{m \in \cM}, D_n \bigr)(x) = j \Bigr\} \Bigr\rvert \enspace .  \]
\end{Definition}
In most situations, it is clear how hold-out rules should be aggregated and there is no ambiguity in discussing hold-out aggregation. However, there is an important exception where both Agghoo and Majhoo can be used.

\begin{Remark}[Two options for binary classification]
\label{rk_surr_loss}
 In binary classification \textup{(}Example~\ref{classif_pbm} with $M = 2$\textup{)}, it is classical to consider classifiers of the form $\mathbb{I}_{f \geq 0}$ where $f \in \parspace_{conv} = \left\{f: \cX \rightarrow \mathbb{R} \right\}$ aims at minimizing a surrogate convex risk associated with the loss 
 $\cpred_{conv}: (y',y) \mapsto \phi [(2y'-1)(2y-1) ]$ with $\phi: \mathbb{R} \rightarrow \mathbb{R}$ convex \cite{Boucheron:2005}.
 Then, given a family of $\parspace_{conv}$-valued learning rules $\bigl(\learnrule_m \bigr)_{m \in \cM}$, one can either apply Agghoo to the surrogate problem and get
 \[ \mathbb{I}_{\ERMag{\cT}\left(\left(\learnrule_m \right)_{m \in \cM}, D_n \right) \geq 0} \enspace , \]
 or apply Majhoo to the binary classification problem and get 
 \[\ERMmv{\cT} \left(\bigl(\mathbb{I}_{\learnrule_m(\cdot) \geq 0} \bigr)_{m \in \cM}, D_n \right) \enspace . \]
\end{Remark}
In the rest of this section, we focus on Agghoo, though much of the following discussion applies also to Majhoo.
\medbreak
Compared to cross-validation rules (Definition~\ref{def_cv}), Agghoo reverses the order between aggregation (majority vote or averaging) and minimization of the risk estimator: instead of averaging hold-out risk estimators before selecting the hyperparameter, the selection step is made first to produce hold-out predictors $\bigl(\ERMho{T} \bigr)_{T \in \cT}$ (given by Definition~\ref{def_ho}) and then an average is taken.

\paragraph{Related procedures} 
%
%
To the best of our knowledge, Agghoo has not been studied theoretically before, 
though it is used in applications \cite{HoyosIdrobo2015,Varoquaux2017}, 
under the name ``CV + averaging'' in \cite{Varoquaux2017}.  
According to \cite{Varoquaux2017}, Agghoo is 
commonly used by the machine learning community 
thanks to the Scikit-learn library \cite{Ped_etal:2011}. 

%
%
A closely related procedure is 
``$K$-fold averaging cross-validation'' (ACV), 
proposed by \cite{Jun_Hu:2015} for linear regression. 
With our general notation, ACV corresponds to averaging 
the $\learnrule_{\widehat{m}_{ho}^T}(D_n)$, which are ``retrained'' on the whole dataset, 
while Agghoo averages the $\learnrule_{\widehat{m}_{ho}^T}(D_n^T)$. 
An advantage of
averaging the rules $\learnrule_{\widehat{m}_{ho}^T}(D_n^T)$ is that they have been selected for their good performance on the validation set $T^c$, unlike the 
$\learnrule_{\widehat{m}_{ho}^T}(D_n)$ whose performance has not been assessed on independent data.
Furthermore, similarly to bagging, using several distinct training sets may result in improvements for unstable methods through a reduction in variance.
Note finally that the theoretical results of \cite{Jun_Hu:2015} on ACV 
are limited to a specific setting, 
and much weaker than an oracle inequality. 

\medbreak

%
A second family of related procedures is averaging 
the chosen \emph{parameters} $\bigl(\widehat{m}_T^{ho}\bigr)_{T \in \cT}$, 
contrary to Agghoo which averages the chosen \emph{prediction rules}. 
This leads to different procedures for learning rules that are not linear functions 
of their parameters. 
This idea has been put forward  
under the name ``bagged cross-validation'' (BCV) \cite{Hall2009} 
---with numerical and theoretical results in the case of 
bandwidth choice in kernel density estimation---, 
and under the name ``efficient $K$-fold cross-validation'' (EKCV) 
\cite{Jun:2016} for the choice of a regularization parameter in 
high-dimensional regression 
---with numerical results only. 
%
Unlike Agghoo, which only depends on the set $\{ \learnrule_m \,\vert\, m \in \cM \}$ of learning rules, EKCV and BCV depend on the parametrization $m \mapsto \learnrule_m$. 
Sometimes, the most natural parametrization does not allow the use of such procedures: 
for example, model dimensions are integers, and averaging them does not make sense.
In contrast, in regression, it is always possible to average the real-valued functions $ \learnrule_m(D_{n_t}) \in \parspace$. 

Even when all procedures are applicable, averaging rules is generally safer than averaging hyperparameters. 
Often in regression, the risk $\risk$ is known to be convex over $\parspace$, so given $t_1,\ldots,t_V \in \parspace$,
\[ \risk \left( \frac{1}{V} \sum_{i = 1} t_i \right) \leq \frac{1}{V} \sum_{i = 1}^V \risk(t_i) \enspace . \]
Hence, averaging regressors (Agghoo) always improves performance compared to selecting a single $t_i$ at random (hold-out).
On the other hand, if $ (t_\theta)_{\theta \in \Theta}$ is a family of elements of $\parspace$ parametrized by a convex set $\Theta$, there is no guarantee in general that the function $\theta \mapsto \risk(t_\theta)$ is convex over $\Theta$. So, for some $\theta_1,\ldots, \theta_V \in \Theta$, it may happen that
\[\risk \left( t_{\frac{1}{V}\sum_{i=1}^V \theta_i} \right) \geq \frac{1}{V} \sum_{i = 1}^V \risk(t_{\theta_i})\enspace . \]
In such a case, it is better to choose one parameter at random (hold-out) that to average them (EKCV or BCV).

\medbreak

%
A third family of related procedures is bagging or subagging applied to hold-out selection 
$D_n \mapsto \ERMho{T}((\learnrule_m)_{m \in \cM}, D_n)$. 
The bagging case has been studied numerically by \cite{Petersen2007}, 
but clearly differs from Agghoo since it relies on bootstrap resamples, 
in which the original data can appear several times. 
Subagging ---which is not explicitly studied in the literature, 
to the best of our knowledge--- is closer to Agghoo, but there is still a slight difference. 
When applying subagging to the hold-out, the sample is divided into three parts: the training part of the bagging subsample, the validation part of the bagging subsample, and the data not in the bagging subsample. With Agghoo, the sample is only divided into two parts.

\subsection{Computational complexity} \label{sec.comp}
In general, for a given value of $V = \lvert \cT \rvert$, both Agghoo ($\ERMag{\cT}$) and CV ($\ERMcv{\cT}$) must compute $V$ hold-out risk estimators over all values of $m \in \cM$. Let $C_{ho}(\cM,n_t,n_v)$ be the average computational complexity of the hold-out, with a training dataset of size $n_t$ and validation dataset of size $n_v$. Then the overall complexity of risk estimation is of order $V \times C_{ho}(\cM,n_t,n_v)$ for both Agghoo and CV. Next, CV must average $V$ risk vectors of length $\lvert \cM \rvert$ and find a single minimum, while Agghoo computes $V$ minima over $m \in \cM$; these operations have similar complexity, of order 
$V \times \lvert \cM \rvert$.
Thus, computing the ensemble aggregated by Agghoo takes about as much time as selecting a learning rule using cross-validation.   
 
A potential difference occurs when evaluating Agghoo and CV on new data.
If there is no fast way to perform aggregation at training time, it is always possible to evaluate each predictor in the ensemble on the new data, and to average the results; then, Agghoo is slower than CV by a factor of order $V$ at test time.

\section{Theoretical results} \label{sec_thms}
The purpose of Agghoo is to construct an estimator whose risk is as small as possible, compared to the (unknown) best rule in the class $(\learnrule_m)_{m \in \cM}$.
This is guaranteed theoretically by proving ``oracle inequalities'' of the form
\begin{equation} \label{or_ineq}
 \mathbb{E}\bigl[ \loss{\ERMag{\cT}} \bigr] \leq C \mathbb{E} \Bigl[ \inf_{m \in \cM} \lossb{\learnrule_m(D_{n})} \Bigr] + \varepsilon_n \enspace ,
\end{equation}
with 
$\varepsilon_n$ negligible compared to the oracle excess risk $\mathbb{E}[\inf_{m \in \cM} \loss{\learnrule_m(D_{n_t})} ]$ 
and $C$ close to $1$. 
Equation~\eqref{or_ineq} then implies that Agghoo performs as well as the best choice of $m \in \cM$, up to the constant $C$.
In the following, we actually prove slightly weaker inequalities that are more natural in our setting.

By definition, Agghoo is an average of predictors chosen by hold-out over the collection $(\learnrule_m)_{m \in \cM}$ . Therefore, when the risk is convex, an oracle inequality \eqref{or_ineq}  can be deduced from an oracle inequality for the hold-out, provided that there exists an integer $n_t \in \{1,\ldots, n-1 \}$ such that
\begin{equation} \label{hyp.T}
 \cT \text{ is independent from } D_n \qquad \text{and} \qquad \forall T \in \cT, \quad |T| = n_t  \enspace .
\end{equation} 
We make this assumption in the rest of the article.
Most cross-validation methods satisfy hypothesis \eqref{hyp.T}, 
including leave-$p$-out, $V$-fold cross-validation (with $n - n_t = n_v =n/V$) and 
Monte-Carlo cross-validation \cite{Arl_Cel:2010:surveyCV}.

In the remainder of this section, we introduce the RKHS setting of interest, 
and prove an oracle inequality for Agghoo without changing the standard estimators or requiring $Y$ to be bounded.

\subsection{Agghoo in regularized kernel regression} \label{sec_rkhs}

Kernel methods such as support vector machines, kernel least squares or $\varepsilon$-regression use a kernel function to map the data $X_i$ into an infinite-dimensional function space, more specifically a reproducing kernel Hilbert space (RKHS) \cite{Sch_Smo:2001, Ste_Chr:2008}. We consider in this section regularized empirical risk minimization using a training loss function $c$, with a penalty proportional to the square norm of the RKHS, to solve the supervised learning problem (defined in Section $2.2$) with loss function $\cpred$. Hence, the contrast $\cfun$ can be written $\cfun(t,(x,y)) = \cpred(t(x),y) := (\cpred \circ t) (x,y)$. We assume that $\cpred$ and $\ctrain$ are convex in their first argument.
 
\begin{Definition}[Regularized kernel estimator]  \label{def_kern} 
 Let $\ctrain: \mathbb{R} \times \mathbb{R} \rightarrow \mathbb{R} $ be convex in its first argument, and let $K: \cX \times \cX \rightarrow \mathbb{R}$ be a positive-definite kernel function.
 Given $\lambda > 0$ and training data $(X_i,Y_i)_{1 \leq i \leq n_t}$,
 define the regularized kernel estimator as
 \[ \learnrule_\lambda(D_{n_t}) = \argmin_{t \in \mathcal{H}} \left\{ P_{n_t} (\ctrain\circ t) + \lambda \hNorm{t}^2 \right\} \enspace  , \]
 where $\mathcal{H}$ is the reproducing kernel Hilbert space induced by $K$. 
By the representer theorem, $\learnrule_\lambda$ can be computed explicitly\textup{:} 
 \begin{gather} \notag 
  \learnrule_\lambda (D_{n_t})(x) = \sum_{j = 1}^{n_t} \widehat{\theta}_{\lambda,j} K(X_j,x) \qquad \text{ where } \\
  \widehat{\theta}_{\lambda} = \argmin_{\theta \in \mathbb{R}^{n_t}} \left\{ \frac{1}{n_t} \sum_{i = 1}^{n_t} \ctrain\left(\sum_{j = 1}^{n_t} \theta_j K(X_j,X_i),Y_i \right) + \lambda \sum_{i = 1}^{n_t} \sum_{j = 1}^{n_t} \theta_i \theta_j K(X_i,X_j) \right\} \enspace . 
  \label{calc_rkhs_reg}
 \end{gather}
\end{Definition}
The loss function $\ctrain$ is used to measure the accuracy of the fit on the training data: for example, taking $\ctrain: (u,y) \mapsto (1-uy)_+$ (the hinge loss) in Definition~\ref{def_kern} corresponds to svm. The loss function $\cpred$ used for risk evaluation may or may not be equal to $\ctrain$. For example, in classification, the 0--1 loss often cannot be used for training for computational reasons, hence a surrogate convex loss, such as the hinge loss, is used instead (see Remark~\ref{rk_surr_loss}), but there is no reason to use the hinge loss for risk estimation and hyperparameter selection.

\medbreak

In Definition~\ref{def_kern}, the hyperparameter of interest is $\lambda$ (we assume that $K$ is fixed).
We show below some guarantees on Agghoo's performance when it is applied to a finite subfamily $\left( \learnrule_\lambda \right)_{\lambda \in \Lambda}$ of the one defined by Definition~\ref{def_kern}.
We first state some useful assumptions.
\\ \\ 
Hypothesis $Comp_C(\cpred,c)$: $\risk_{\ctrain}: t \mapsto P(\ctrain\circ t)$ and $\risk_{\cpred}$ have a common minimum $s \in \argmin_{t \in \parspace} \risk_{\ctrain}(t) \cap \argmin_{t \in \parspace} \risk_{\cpred}(t)$ and for any $t \in \parspace$,
$ \risk_{\ctrain}(t) - \risk_{\ctrain}(s) \leq C\left[ \risk_{\cpred}(t) - \risk_{\cpred}(s)\right]$.
\\ \\
Note that $Comp_1(\cpred,c)$ is always satisfied when $\cpred = c$. When $g \neq c$, some hypothesis relating $c$ and $g$ is necessary anyway for Definition~\ref{def_kern} to be of interest, if only to ensure consistency (asymptotic minimization of the risk)  for some sequence of hyperparameters $(\lambda_n)_{n \in \mathbb{N}}$.

\medbreak

In addition, some information about the evaluation loss $\cpred$ helps to obtain an oracle inequality \eqref{or_ineq} with a smaller remainder term $\varepsilon_n$.    
\\ \\
Hypothesis $SC_{\cvxct,\slope}$: Let $\xloss{u} = \mathbb{E}[\cpred(u,Y) | X] - \inf_{v \in \mathbb{R}} \mathbb{E}[\cpred(v,Y) | X]$.
The triple $(\cpred, X,Y)$ satisfies $SC_{\cvxct,\slope}$ if and only if, for any $u,v \in \mathbb{R}$,
\begin{equation} 
\label{hyp.SC} 
 \mathbb{E}[(\cpred(u,Y) - \cpred(v,Y))^2 | X] \leq \bigl[ \cvxct \vee (\slope |u-v|) \bigr] [\xloss{u} + \xloss{v}]. 
\end{equation}
\\ 
\\
For example, in the case of median regression, that is, $g(u,y) = |u-y|$, hypothesis $SC_{\rho,\nu}$ holds whenever there is a uniform lower bound on the concentration of $Y$ around $\bayes(X)$, as shown by the following proposition.

\begin{proposition} \label{prop_sc}
 Let $\cpred(u,y) = |u-y|$ for all $u,y \in \mathbb{R}$. For any $x \in \cX$, let $F_x$ be the conditional cumulative distribution function of $Y$ knowing $X = x$. 
 Assume that, for any $x \in \cX$, $F_x$ is continuous with a unique median $\bayes(x)$ and that  there exists 
 $a(x) > 0, b(x) > 0$ such that
 \begin{equation} \label{hyp_loc_incr_F}
 \forall u \in \mathbb{R}, \qquad 
 \Bigl\lvert F_x(u) - F_x \bigl( \bayes(x) \bigr) \Bigr\rvert \geq a(x) \Bigl[ \bigl\lvert u-\bayes(x) \bigr\rvert \wedge b(x) \Bigr] \enspace .
 \end{equation}
For instance, this holds true if $\frac{dF_x}{du} \geq a(x) \mathbb{I}_{|u-\bayes(x)| \leq b(x)}$ for every $x \in \cX$. Let
\begin{align*}
 a_m = \inf_{x \in \cX} \bigl\{ a(x) \bigr\} 
\qquad \text{and} \qquad 
\mu_m = \inf_{x \in \cX} \bigl\{a(x)b(x) \bigr\} 
\enspace  .
\end{align*}
If $a_m > 0$ and $\mu_m > 0$, then $(\cpred,X,Y)$ satisfies $SC_{\frac{4}{a_m}, \frac{2}{\mu_m}}$.
\end{proposition}
Proposition~\ref{prop_sc} is proved in Appendix~\ref{sec.proof_sc}.
We can now state our first main result.

\begin{Theorem} \label{rkhs_thm}
  Let $\Lambda \subset \mathbb{R_+^*}$ be a finite grid. Using the notation of Definition~\ref{agcv_def}, let $\ERMag{\cT}$ be the output of Agghoo, applied to the collection $(\learnrule_\lambda)_{\lambda \in \Lambda}$ given by Definition~\ref{def_kern}. 
 Assume that $\lambda_m = \min \Lambda > 0$ and $\kappa = \sup_{x \in \cX} K(x,x) < +\infty$.
 Assume that $Comp_C(\cpred,\ctrain)$ holds for a constant $C > 0$ and that $(\cpred,X,Y)$ satisfies $SC_{\cvxct,\slope}$ with constants $\rho \geq 0, \nu \geq 0$.
 Assume  that $\ctrain$ and $\cpred$ are convex and Lipschitz in their first argument, with Lipschitz constant less than $L$. Assume also that $n_v \geq 100$ and $3 \leq \lvert \Lambda \rvert \leq \mathrm{e}^{\sqrt{n_v}}$.
Then, for any $\theta \in (0;1]$,
\begin{equation} 
\label{rkhs_thm_eq2}
\begin{split}
  (1 - \theta ) \mathbb{E}\Bigl[\ell \bigl(s, \ERMag{\cT} \bigr) \Bigr] 
 &\leq ( 1+\theta ) \mathbb{E}\Bigl[\min_{\lambda \in \Lambda} \ell \bigl(s,\learnrule_\lambda(D_{n_t}) \bigr) \Bigr]
 \\ 
 &\hspace*{-1cm}+ \max \left\{18 \cvxct \frac{\log\bigl(n_v \lvert \Lambda \rvert\bigr)}{\theta n_v}, b_1 \frac{\log^2 \bigl(n_v\lvert \Lambda \rvert\bigr)}{\theta^3 \lambda_m n_v^2}, b_2 \frac{\log^{\frac{3}{2}}\bigl(n_v\lvert \Lambda \rvert\bigr)}{\theta \lambda_m n_v \sqrt{n_t}} \right\} \enspace , 
\end{split}
\end{equation}
where $b_1,b_2$ do not depend on $n_v,n_t,\lambda_m$ or $\theta$ but only on $\kappa,L,\slope$ and $C$.
\end{Theorem}
Theorem~\ref{rkhs_thm} is proved in Appendix~\ref{sec_proof_rkhs} as a consequence of a result valid in the general framework of Section~\ref{risk_minim} (Theorem~\ref{agcv_mean}).
It shows that $\ERMag{\cT}$ satisfies an oracle inequality of the form \eqref{or_ineq}, with $\learnrule_\lambda(D_{n_t})$ instead of $\learnrule_\lambda(D_n)$ on the right-hand side of the inequality. The fact that $D_{n_t}$ appears in the bound instead of $D_n$ is a limitation of our result, but it is natural since predictors aggregated by Agghoo are only trained on part of the data. In most cases, it can be expected that $\loss{\learnrule_\lambda(D_{n_t})}$ is close to $\loss{\learnrule_\lambda(D_{n})}$ whenever $\frac{n_t}{n}$ is close to $1$. 

The assumption that $K$ is bounded is mild. For instance, popular kernels such as Gaussian kernels, $(x,x') \mapsto \exp [-\Norm{x-x'}^2/(2h^2)]$ for some $h > 0$, or Laplace kernels, $(x,x') \mapsto \exp (-\Norm{x-x'}/h)$ for some $h > 0$, are bounded by $\kappa = 1$.

Taking $\lvert \cT \rvert = 1$ in Theorem~\ref{rkhs_thm} yields a new oracle inequality for the hold-out. 
Oracle inequalities for the hold-out have already been proved in a variety of settings (see \cite{Arl_Cel:2010:surveyCV} for a review), and used to obtain adaptive rates
in regularized kernel regression \cite{Ste_Chr:2008}. 
However, this work has mostly been accomplished under the assumption that the contrast $\cfun \left( \learnrule_\lambda(D_n),(X,Y) \right)$ is bounded uniformly (in $n$, $D_n$ and $\lambda \in \Lambda$) by a constant. 
If this constant increases with $n$, bounds obtained in this manner may worsen considerably.
As many ``natural'' regression procedures ---including regularized kernel regression (Definition~\ref{def_kern})--- fail to satisfy such bounds, some theoreticians introduce ``truncated'' versions of standard procedures \cite{Ste_Chr:2008}, 
but truncation has no basis in practice. Theorem~\ref{rkhs_thm} avoids these complications.

\medbreak

In order to be satisfactory, Theorem~\ref{rkhs_thm} should prove that Agghoo performs asymptotically as well as the best choice of $\lambda \in \Lambda$,
at least for reasonable choices of $\Lambda$.
This is the case whenever the maximum in Equation~\eqref{rkhs_thm_eq2} is negligible with respect to the oracle excess risk $\mathbb{E}[\min_{\lambda \in \Lambda} \loss{\learnrule_\lambda(D_{n_t})} ]$ as $n \rightarrow + \infty$. 
This depends on the range $[\lambda_m;+\infty)$ in which the hold out is allowed to search for the optimal $\lambda$. On the one hand, it is desirable that this interval be wide enough to contain the true optimal value. On the other hand, if $\lambda_m = 0$, then inequality \eqref{rkhs_thm_eq2} becomes vacuous. We now provide precise examples where Theorem~\ref{rkhs_thm} applies with a remainder term in Equation~\eqref{rkhs_thm_eq2} that is negligible relative to the oracle excess risk. 

Take the example of median regression, in which $\ctrain(u,y) = \cpred(u,y) = |u-y|$. Then $Comp_1(\cpred,\ctrain)$ holds trivially.
Make also the same assumptions as in Proposition~\ref{prop_sc}, which ensures that $SC_{\rho,\nu}$ holds for some 
finite values of $\rho$ and $\nu$. Theorem~\ref{rkhs_thm} therefore applies as long as the kernel $K$ is bounded 
and $\lambda_m > 0$. Choose $n_v = n_t = \frac{n}{2}$ and $\Lambda$ of cardinality at most polynomial in $n$ (which
is sufficient in theory and in practice).
Then \cite[Theorem 9.6]{Ste_Chr:2008} proves the consistency of $\learnrule_{\lambda_n}(D_n)$ as $n \rightarrow +\infty$, provided that $\lambda_n^2 n \rightarrow + \infty$.
This suggests choosing $\lambda_m = 1/\sqrt{n_t}$, in which case the remainder term of Equation~\eqref{rkhs_thm_eq2} is of order $(\log n)^{3/2} / n$, which is negligible relative to nonparametric convergence rates in median regression.

In order to have a more precise idea of the order of magnitude of the oracle excess risk, let us consider median regression with a Gaussian kernel. Under some assumptions, one of which coincides with Proposition~\ref{prop_sc}, \cite[Corollary 4.12]{eberts2013} shows that taking $\lambda_n = \frac{c_1}{n}$ leads to rates of order $n^{-\frac{2\alpha}{2\alpha + d}}$, where $d \in \mathbb{N}$ is the dimension of $\cX$ and $\alpha > 0$ is the smoothness of $\bayes$. 
Therefore, taking $\lambda_m = 1 / n_t$ in Theorem~\ref{rkhs_thm}, the remainder term of Equation~\eqref{rkhs_thm_eq2} is at most of order $(\log n)^{3/2} / \sqrt{n}$, hence negligible relative to the above risk rates as soon as $2\alpha < d$.

\medbreak

Theorem~\ref{rkhs_thm} can handle situations where $\cpred$ is different from the training loss $\ctrain$, provided that $Comp(\cpred,\ctrain)$ holds true. Such situations arise for instance in the case of support vector regression \cite[Chapter 9]{Sch_Smo:2001}, which uses for training Vapnik's $\varepsilon$-insensitive loss $c^{eps}_\varepsilon(u,y) = (|u-y| - \varepsilon)_+$. This loss depends on a parameter $\varepsilon$, the choice of which is usually motivated by a tradeoff between sparsity and prediction accuracy \cite{Sch_Smo:2001}. Therefore, some other loss is typically used to measure predictive performance, independently of $\varepsilon$.
We state one possible application of Theorem~\ref{rkhs_thm} to this case, as a corollary.

\begin{corollary}[$\varepsilon$-regression] \label{eps_reg} 
Let $\ctrain= c^{eps}_\varepsilon: (u,y) \mapsto (|y-u| - \varepsilon)_+$
be Vapnik's $\varepsilon$-insensitive loss and assume that the evaluation loss is $\cpred = c_0^{eps}: (u,y) \mapsto |u -y|$. Assume that for every $x$ the conditional distribution of $Y$ given $X = x$ has a  unimodal density with respect to the Lebesgue measure, symmetric around its mode. Introduce the robust noise parameter\textup{:} 
\begin{equation}
\begin{split}
\sigma &= \sup_{x \in \cX} \left\{  \inf \left\{y \in \mathbb{R} \,\Big\vert\, \mathbb{P}(Y \leq y \,\vert\, X = x) \geq \frac{3}{4}  \right\} \right. \\ 
 &\qquad\qquad  
 \left. - \sup \left\{y \in \mathbb{R} \,\Big\vert\, \mathbb{P}(Y \leq y \,\vert\, X = x) \leq \frac{1}{4} \right\} \right\} 
 \enspace . 
\end{split}
\label{def_sigma}
\end{equation}
Then, applying Agghoo to a finite subfamily $(\learnrule_\lambda)_{\lambda \in \Lambda}$ of the rules given by Definition~\ref{def_kern} with $\ctrain= c^{eps}_\varepsilon$ and a kernel $K$ such that $\NormInfinity{K} \leq 1$ yields the following oracle inequality. Assuming $n_v \geq 100$ and $3 \leq \lvert \Lambda \rvert \leq \mathrm{e}^{\sqrt{n_v}}$, for any $\theta \in (0;1]$,
\begin{align*}
(1 - \theta ) \mathbb{E}\Bigl[\lossb{\ERMag{\cT}}\Bigr] 
 &\leq ( 1+\theta )\mathbb{E}\Bigl[\min_{\lambda \in \Lambda} \lossb{\learnrule_\lambda(D_{n_t})} \Bigr] \\ 
 &+ \max \left\{ 72\sigma \frac{\log \bigl( n_v \lvert \Lambda \rvert \bigr)}{\theta n_v}  \, , \, 
 b_1 \frac{\log^2 \bigl( n_v\lvert \Lambda \rvert \bigr)}{\theta^3 \lambda_m n_v^2} \, , \, b_2 \frac{\log^{\frac{3}{2}} \bigl( n_v\lvert \Lambda \rvert \bigr)}{\theta \lambda_m n_v \sqrt{n_t}} \right\} 
 \enspace ,
\end{align*}
where $b_1$ and $b_2$ are absolute constants.
\end{corollary}
Corollary~\ref{eps_reg} is proved in Appendix~\ref{sec.proof_cor_eps_reg}.

When $\varepsilon = 0$, 
$\varepsilon$-regression becomes median regression, which is discussed above. The oracle inequality of Corollary~\ref{eps_reg} is then the same as that given by Theorem~\ref{rkhs_thm} and Proposition~\ref{prop_sc}. Assumptions of unimodality and symmetry allow to give more explicit values of $a_m$ and $\mu_m$ in terms of $\sigma$. 
When $\varepsilon > 0$, 
the unimodality and symmetry assumptions are used to prove hypothesis $Comp_C(\cpred,c)$.

\subsection{Classification}

Loss functions are not all convex. When convexity fails, the aggregation procedure should be revised. 

In classification, Majhoo is a possible solution (see Definition~\ref{agcv_def_classif}). By Proposition~\ref{maj_vote} in Appendix~\ref{sec_classif}, majority voting satisfies a kind of ``convexity inequality'' with respect to the 0--1 loss; as a result, oracle inequalities for the hold-out imply oracle inequalities for majhoo.

Hold-out for binary classification with 0--1 loss has been studied by Massart \cite{Mas:2003:St-Flour}.
In that work, Massart makes an assumption which is closely related to margin hypotheses, such as the Tsybakov noise condition \cite{Mam_Tsy:1999} which we consider here. This approach allows to derive the following theorem.

\begin{Theorem}
\label{thm_classif}
Consider the classification setting described in Example~\ref{classif_pbm} with $M = 2$ classes \textup{(}binary classification\textup{)}.
Let $(\learnrule_m)_{m \in \cM}$ be a collection of learning rules 
and $\mathcal{T}$ a collection of training sets 
satisfying assumption~\eqref{hyp.T}. 

Assume that there exists $\beta \geq 0$ and $r \ge 1$ such that for $\xi = (X,Y)$ with distribution $P$,
\begin{equation}\label{hyp.MA} \tag{MA}
\forall h>0, \qquad 
\mathbb{P} \bigl( \bigl\lvert 2\eta(X) - 1 \bigr\rvert \leq h \bigr) \leq r h^{\beta}
\end{equation}
where $\eta(X) := \mathbb{P} (Y = 1 \,\vert\, X )$. 
Then, we have  
\[ 
\bE \Bigl[ \lossb{\ERMmv{\cT}} \Bigr] 
\leq 3 \bE \left[ \inf_{m \in \cM} \lossb{\learnrule_m(D_{n_t})} \right] 
 + \frac{ 29 r^{\frac{1}{\beta + 2}}  \log \bigl( \mathrm{e} \lvert \cM \rvert \bigr)}{n_v^{\frac{\beta + 1}{\beta + 2}}}  
 \enspace .
\]
\end{Theorem}

Theorem~\ref{thm_classif} is proved in Appendix~\ref{sec_classif}. 
It shows that $\ERMmv{\cT}$, like $\ERMag{\cT}$, satisfies an oracle inequality of the form \eqref{or_ineq} with $\learnrule_\lambda(D_{n_t})$ instead of $\learnrule_\lambda(D_n)$.
Tsybakov's noise condition \eqref{hyp.MA} only depends on the distribution of $(X,Y)$ and not on the collection of learning rules. 
It is a standard hypothesis in classification, under which ``fast'' learning rates ---faster than $n^{-1/2}$--- are attainable \cite{Tsy:2004}. 
In contrast with the results of Section~\ref{sec_rkhs}, 
that are valid for various losses but only for a specific type of learning rule, 
Theorem~\ref{thm_classif} holds true for \emph{any} family of classification rules.

The constant $3$ in front of the oracle excess risk can be replaced by any constant larger than $2$, 
at the price of increasing the constant in the remainder term, as can be seen from the proof (in Appendix~\ref{sec_classif}). 
However, our approach cannot yield a constant lower than $2$, 
because we use Proposition~\ref{maj_vote} instead of a convexity argument, since the 0--1 loss is not convex.

\section{Numerical experiments} \label{sec_simus}

This section investigates how Agghoo and Majhoo's performance vary with their parameters $V$ and $\tau = \frac{n_t}{n}$, and how it compares to CV's performance at a similar computational cost ---that is, for the same values of $V$ and $\tau$. 
Two settings are considered, corresponding to Corollary~\ref{eps_reg} and Theorem~\ref{thm_classif}.

\subsection{$\varepsilon$-regression} \label{eps_reg_simu}
Consider the collection $(\learnrule_\lambda)_{\lambda \in \Lambda}$ of regularized kernel estimators (see Definition~\ref{def_kern}) with loss function $c^{eps}_\varepsilon(u,y) = (|u-y|- \varepsilon)_+$ and Gaussian kernel $K(x,x') = \exp [- (x-x')^2 / (2h^2) ]$ over $\cX = \mathbb{R}$.

\paragraph{Experimental setup}
Data $(X_1,Y_1), \ldots, (X_n,Y_n)$ are independent, 
with $X_i \sim \mathcal{N}(0,\pi)$, $Y_i = \bayes(X_i) + Z_i$,
with $Z_i \sim \mathcal{N}(0,1/2)$ independent from $X_i$. The regression function is
$ \bayes(x) = \mathrm{e}^{\cos(x)}$, the kernel parameter is $h = \frac{1}{2}$ and the threshold for the $\varepsilon$-insensitive loss is $\varepsilon = \frac{1}{4}$. 
Agghoo is applied to  $\left(\learnrule_\lambda \right)_{\lambda \in \Lambda}$ over the grid $\Lambda = \{ \frac{2^{j-1}}{500 n_t} \,\vert\, 0 \leq j \leq 17  \} $, corresponding to the grid $\{ \frac{500}{2^j} \,\vert\, 0 \leq j \leq 17 \}$ over the cost parameter $C = \frac{1}{2\lambda n_t}$. 
Risk estimation is performed
using $L^1$ loss $\cpred(u,y) = \lvert u - y \rvert$.
%
%
Agghoo and CV training sets $T \in \cT$ are chosen independently and uniformly 
among the subsets of $\{1, \ldots, n\}$ with cardinality $\floor{\tau n}$,  
for different values of $\tau$ and~$V = \lvert \cT \rvert$; 
hence, CV corresponds to what is usually called ``Monte-Carlo CV'' \cite{Arl_Cel:2010:surveyCV}. 
%
%
Each algorithm is run on $1000$  
independent samples of size $n=500$, 
and independent test samples of size $1000$ are used 
for estimating the $L^1$ excess risks 
$\ell(\bayes,\ERMag{\cT})$, $\ell(\bayes, \ERM{\cT}^{\mathrm{cv}} ) $ 
and the oracle excess risk $\inf_{\lambda \in \Lambda} \ell(\bayes, \learnrule_\lambda(D_n) )$. 
Expectations of these quantities are estimated by taking an average over the $1000$ samples; 
we also compute standard deviations for these estimates, 
which are not shown on Figure~\ref{eps_svm_simu} since they are all smaller than $2.7 \%$ 
of the estimated value, so that most visible differences on the graph are significant.

\begin{figure}
 \vspace*{-2cm} 
\includegraphics[scale = 0.6]{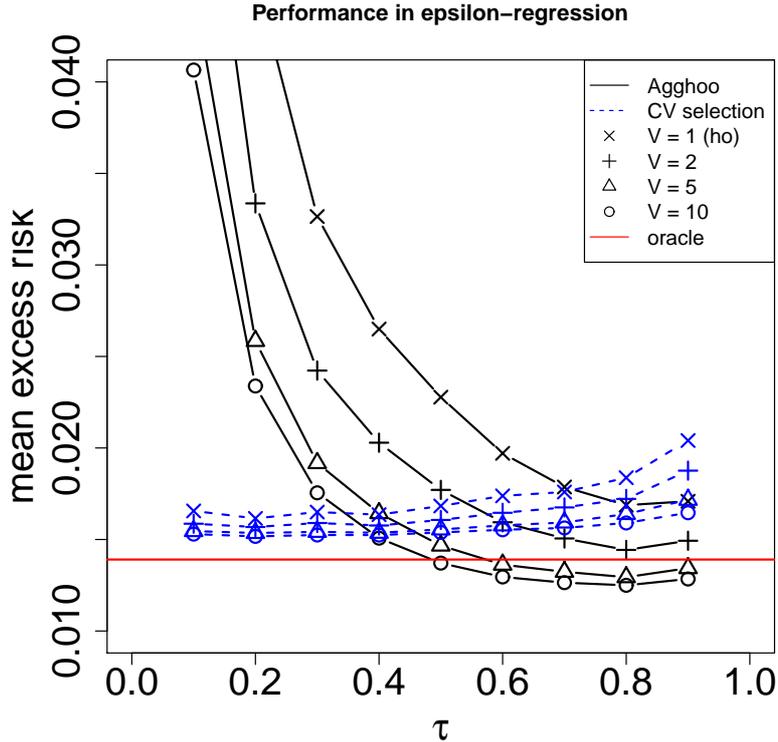} 
\vspace*{-0.5cm} 
  \caption{\label{eps_svm_simu} Performance of Agghoo and CV 
for $\varepsilon$-regression 
}
\end{figure}

\paragraph{Results} are shown on Figure~\ref{eps_svm_simu}.
The performance of Agghoo strongly depends on both $\tau$ and $V$. 
For a fixed $\tau$, increasing $V$ improves significantly the performance of the resulting estimator.  
Most of the improvement occurs between $V = 1$ and 
$V = 5$, and taking $V$ much larger seems useless ---at least for $\tau \geq 0.5$---, 
a behavior previously observed for CV \cite{Arl_Ler:2012:penVF:JMLR}. 
For a fixed $V$, the risk strongly decreases when $\tau$ increases from $0.1$ to $0.5$, decreases slowly over the interval $[0.5;0.8]$ and seems to rise for $\tau > 0.8$. 
It seems that $\tau\in [0.6,0.9]$ yields the best performance, 
while taking $\tau$ close to $0$ should clearly be avoided (at least for $V \leq 10$). 
Taking $V$ large enough, say $V=10$, makes the choice of $\tau$ less crucial: 
a large region of values of $\tau$ yield (almost) optimal performance. We do not know whether taking $V$ larger can make the performance of Agghoo with $\tau \leq 0.4$ close to the optimum.

As a function of $\tau$, the risk of CV behaves quite differently from Agghoo's. The performance does not degrade significantly when $\tau$ is small. The optimum is located at $\tau = 0.2$, which is much smaller than for Agghoo. A possible explanation is that the regressors produced by cross-validation are all trained on the whole sample, so that $\tau$ only impacts risk estimation. Furthermore, additional simulations show, as expected, that higher values of $\tau$ ($\tau  = 0.8$ or $\tau = 0.9$) improve \emph{risk estimation} while degrading the \emph{hyperparameter selection} performance.
Compared to Agghoo, CV's performance depends much less on $V$: only $V=2$ appears to be significantly worse than $V \geq 5$. 

%
Let us now compare Agghoo and CV. For a given $\tau$, Agghoo performs much better than the hold-out. 
This is not surprising and confirms that considering several data splits is always useful. 
For fixed $(\tau,V)$ with $\tau \geq 0.5$, Agghoo does significantly better than CV if $V \geq 5$, mostly worse if $V=1$, 
and they yield similar performance for $V=2$. 
When both parameters are well chosen, Agghoo can outperform the oracle, which is possible because Agghoo involves aggregation. 
Cross-validation, which is a pure selection method, naturally cannot beat the oracle.
Overall, if the computational cost of $V=10$ data splits is not prohibitive, 
Agghoo with optimized parameters ($V=10$, $\tau \in [0.6 , 0.9]$) 
clearly improves over CV with optimized parameters ($V=10$, $\tau=0.2$). The same holds with $V = 5$.
This advocates for the use of Agghoo instead of CV, unless 
we have to take $V < 5$ for computational reasons.

\paragraph{Computational complexity} 
By Equation~\eqref{calc_rkhs_reg}, regularized kernel regressors can be represented linearly by vectors of length $n_t$, therefore the aggregation step can be performed at training time by averaging these vectors. The complexity of this aggregation is at most $\mathcal{O}(V \times n_t )$. In general, this is negligible relative to the cost of computing the hold-out, as simply computing the kernel matrix requires $n_t(n_t+1) / 2$ kernel evaluations. Therefore, the aggregation step does not affect much the computational complexity of Agghoo, so the conclusion of Section~\ref{sec.comp} that Agghoo and CV have similar complexity applies in the present setting. 

Evaluating Agghoo and CV on new data $x \in \cX$ also takes the same time in general, as both are computed by evaluating the expression $\sum_{j = 1}^{n_t} \theta_j K(X_j,x)$ with a pre-computed value of $\theta$. A potential difference occurs when the $\widehat{\theta}_\lambda$ ---given by Definition~\ref{def_kern}, Equation~\eqref{calc_rkhs_reg}--- are sparse: aggregation increases the number of non-zero coefficients, so evaluating $\ERMag{\cT}$ on new data can be slower than evaluating $\ERMcv{\cT}$ if the implementation is designed to take advantage of sparsity.

%


\subsection{$k$-nearest neighbors classification} \label{sec.simus.kNN}
Consider the collection $ (\learnrule_k^{\mathrm{NN}})_{k \geq 1, \, k \text{ odd}}$ 
of nearest-neighbors classifiers ---assuming $k$ is odd to avoid ties--- 
on the following binary classification problem. 

\paragraph{Experimental setup} 

Data $(X_1,Y_1), \ldots, (X_n,Y_n)$ are independent, 
with $X_i$ uniformly distributed over $\cX = [0,1]^2$ 
and 
\begin{gather*}
\mathbb{P} (Y_i = 1 \,\vert\, X_i) 
= \sigma\left(\frac{g(X_i) - b}{\lambda}\right) 
\\
\text{where } 
\forall u,v \in \mathbb{R}, \qquad 
\sigma(u) = \frac{1}{1 + \mathrm{e}^{-u}} 
\quad \text{and} \quad 
g(u,v) = \mathrm{e}^{-(u^2 + v)^3} + u^2 + v^2
\enspace , 
\end{gather*}
$b = 1.18$ and $\lambda = 0.05$. 
The Bayes classifier is $\bayes: x \mapsto \un_{g(x) \geq b}$ 
and the Bayes risk, computed numerically using the scipy.integrate python library, 
is approximately equal to $0.242$. 
%
%
%
Majhoo (the classification version of Agghoo, see Definition~\ref{agcv_def_classif}) and CV are used with the collection $(\learnrule_k^{\mathrm{NN}})_{k \geq 1, \, k \text{ odd}}$
and ``Monte Carlo'' training sets as in Section~\ref{eps_reg_simu}. An experimental procedure similar to the one of Section~\ref{eps_reg_simu} is used to evaluate the performance of Agghoo and to compare it with Monte-Carlo cross-validation. Standard deviations of the excess risk were computed; they are smaller than $3.6 \%$ of the estimated value. 

\begin{figure}
 \vspace*{-2cm} 
\includegraphics[scale = 0.6]{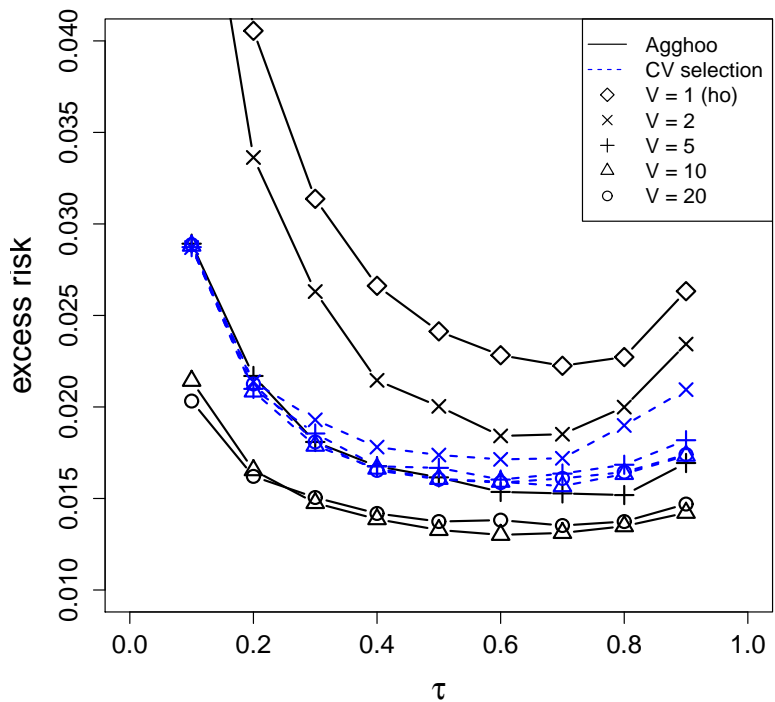} 
\vspace*{-0.5cm} 
  \caption{\label{AgghooVsCV} Classification performance of Majhoo and CV 
for the $k$-NN family 
}
\end{figure}

\paragraph{Results} are shown on Figure~\ref{AgghooVsCV}.
They are similar to the regression case (see Section~\ref{eps_reg_simu}), with a few differences. First, Agghoo does not perform better than the oracle. In fact, all methods considered here remain far from the oracle, which has an excess risk around $0.0034 \pm 0.0004$; both Agghoo and CV have excess risks at least 4 times larger. 
Second, risk curves as a function of $\tau$ for Agghoo are almost $U$-shaped, with a significant rise of the risk for $\tau > 0.6$. Therefore, less data is needed for training, compared to Section~\ref{eps_reg_simu}. The optimal value of $\tau$ here is $0.6$, at least for some values of $V$, up to statistical error. Third, the performance of CV as a function of $\tau$ has a similar U-shape, which makes the comparison between Agghoo and CV easier. 
For a given $\tau$, Agghoo performs significantly better if $V \geq 10$, while CV performs significantly better if $V = 2$; the difference is mild for $V = 5$.

\paragraph{Computational complexity}
As said in Section~\ref{sec.comp}, the complexity of computing the optimal parameters for CV ($\hat{k}_{\cT}^{cv}$) is the same as for Majhoo ($ (\hat{k}_{T}^{ho})_{T \in \cT}$). Here, there is no simple way to represent the aggregated estimator, so aggregation may have to be performed at test time.
In that case, the complexity of evaluating Majhoo on new data is roughly $V$ times greater than for CV, as explained in Section~\ref{sec.comp} for Agghoo.

\section{Discussion} \label{sec.concl}

%
Theoretical and numerical results of the paper 
show that Agghoo can be used safely in RKHS regression, at least when its parameters are properly chosen; 
$V \geq 10$ and $\tau = 0.8$ seem to be safe choices. 
A variant, Majhoo, can be used in supervised classification with the 0--1 loss, with a general guarantee on its performance (Theorem~\ref{thm_classif}). 
Experiments show that Agghoo actually performs much better than what the upper bounds of Section~\ref{sec_thms} suggest, with a significant improvement over cross-validation except when $V < 5$ splits are used. Proving theoretically that Agghoo can improve over CV is an open problem that deserves future works.  
%
%

%
Since Agghoo and CV have the same training computational cost for fixed $(V,\tau)$,
Agghoo ---with properly chosen parameters $V,\tau$--- should be preferred to CV, 
unless aggregation is undesirable for some other reason, such as interpretability of the predictors, or computational complexity at test time. 
%

%
\medbreak
Our results can be extended in several ways. 
%
%
First, our theoretical bounds directly apply to subagging hold-out,  
which also averages several hold-out selected estimators. 
The difference is that, in subagging, the training set size is $n-p-q$ and the validation set size is $q$, 
for some $q \in \{1, \ldots, n-p-1\}$, 
leading to slightly worse bounds than those we obtained for Agghoo (at least if $\mathbb{E}\left[ \loss{\learnrule_m(D_n)} \right]$ decreases with $n$). The difference should not be large in practice, if $q$ is well chosen. 

Oracle inequalities can also be obtained for Agghoo in other settings, as a consequence of our general theorems \ref{hold_out_gen2} and \ref{agcv_mean} in Appendix~\ref{app.general-thms}.  
%
%
%
%
%
%

\appendix

\section{General Theorems} \label{app.general-thms}

We need the following hypothesis, defined for two functions $w_i : \mathbb{R}_+ \rightarrow \mathbb{R_+}$, $i \in \{1;2\}$ and a family $(t_m)_{m \in \cM} \in \parspace^{\cM}$. 
\\ \\ 
Hypothesis $H(w_1,w_2,(t_m)_{m \in \cM} )$: 
$w_1$ and $w_2$ are non-decreasing, 
and for any $(m,m') \in \cM^2$, some $c_{m'}^m \in \mathbb{R}$ exists such that, 
for all $k \geq 2$, 
\begin{equation*} 
\begin{split}
P \Bigl( \bigl\lvert \cfun(t_m)- \cfun (t_{m'}) - c_{m'}^m \bigr\rvert^k \Bigr)
&\leq k! \Bigl[ w_1\bigl(\sqrt{\loss{t_m}}\bigr) + w_1(\sqrt{\loss{t_{m'}}}) \Bigr]^2 
\\
&\qquad \times \Bigl[ w_2\bigl(\sqrt{\loss{t_m}} \bigr) + w_2 \bigl( \sqrt{\loss{t_{m'}}} \bigr) \Bigr]^{k-2} 
\enspace .  
\end{split}
\end{equation*}
%
%
This hypothesis is similar to those used by Massart \cite{Mas:2003:St-Flour} to study the hold-out and empirical risk minimizers.
However, unlike \cite{Mas:2003:St-Flour}, we intend to go beyond the setting of bounded risks.

We also need the following definition. 
\begin{Definition} \label{def_delta}
Let $w: \mathbb{R}_+ \rightarrow \mathbb{R}_+$ and $r\in \mathbb{R}_+$.
Let
  \[ \delta(w,r) =  \inf \left\{ \delta \geq 0 : \forall x \geq \delta, w(x) \leq r x^2 \right\},\]
  with the convention $\inf \emptyset = + \infty$.
\end{Definition}

\begin{Remark}
\label{rmk_delta}
\begin{itemize}
 \item  If $r > 0$ and $x \mapsto \frac{w(x)}{x}$ is nonincreasing, then $\delta(w,r)$ is the unique solution to the equation $\frac{w(x)}{x} = r x$.
 \item $r \mapsto \delta(w,r)$ is nonincreasing.
 \item If $w(x) = c x^{\beta}$ for $c > 0$ and $\beta \in [0;2)$, then $\delta(w,r) = \left(\frac{c}{r}\right)^\frac{1}{2-\beta}$.
\end{itemize} 
\end{Remark}

\subsection{Theorem statements}
We can now state two general theorems from which we deduce all the theoretical results of the paper. 
The first theorem is a general oracle inequality for the hold-out. 
\begin{Theorem} \label{hold_out_gen2} 
Let $(t_m)_{m \in \mathcal{M}}$ be a finite collection in $\parspace$,
 and \[ \widehat{m} \in \argmin_{m \in \cM} P_{n_v} \cfun (t_m,\cdot) \enspace . \]
Assume that $H(w_1,w_2,(t_m)_{m \in \cM} )$ holds true. 
Let $x > 0$. 
Then, with probability larger than $1 - \mathrm{e}^{-x}$, for any $\theta \in (0;1]$, we have 
\begin{align*}
 \left(1 - \theta \right) \loss{t_{\widehat{m}}} &\leq \left(1 + \theta \right) \min_{m \in \cM} \loss{t_m} + \sqrt{2} \theta \delta^2 \left(w_1, \frac{\theta}{2}\sqrt{\frac{n_v}{x + \log \lvert \cM \rvert}} \right) \\
 &\qquad + \frac{\theta^2}{2} \delta^2 \left(w_2,\frac{\theta^2}{4}\frac{n_v}{x + \log \lvert \cM \rvert}\right). \numberthis \label{eq_plus2}
\end{align*}
If in addition, the two functions $x \mapsto \frac{w_{\ind}(x)}{x}$, $\ind=1,2$, are nonincreasing,
  then for any $x > 0$, 
  with probability larger than $1 - \mathrm{e}^{-x}$,  for all $\theta \in (0;1]$, we have 
\begin{align}
 (1- \theta)\loss{t_{\widehat{m}}} &\leq (1+\theta) \min_{m \in \cM} \loss{t_m} + \delta^2(w_1,\sqrt{n_v}) \left[ \theta + \frac{2(x + \log \lvert \cM \rvert)}{\theta} \right] 
\\ 
 &\qquad + \delta^2(w_2,n_v)\left[ \theta + \frac{(x + \log \lvert \cM \rvert)^2}{\theta} \right] 
 \enspace . \label{eq_2plus2}
\end{align}
\end{Theorem}

Using Theorem~\ref{hold_out_gen2}, we prove the following general oracle inequality for Agghoo. 
\begin{Theorem}  \label{agcv_mean} 
Assume that the hyperparameter space $\parspace$ is convex 
and that the risk $\risk$ is convex.
Let $( \learnrule_m )_{m \in \cM}$ be a finite collection of learning rules of size $\lvert \cM \rvert \geq 3$.
Let $\ERMag{\cT}$ be an Agghoo estimator, according to Definition~\ref{agcv_def}, with $\cT$ satisfying assumption~\eqref{hyp.T}.
Assume that $\widehat{w}_{1,1}, \widehat{w}_{1,2}$ are $D_{n_t}$-measurable random functions such that almost surely, $H\bigl(\widehat{w}_{1,1}, \widehat{w}_{1,2}, (\learnrule_m(D_{n_t}))_{m \in \cM} \bigr)$ holds true. Assume also that for $i \in \{1,2\}$,
 $x \mapsto \frac{\widehat{w}_{1,i}(x)}{x}$ is non-increasing. 
Then for any $\theta \in (0;1]$,
\begin{equation} \label{eq_12}
 (1-\theta) \mathbb{E} \Bigl[ \lossb{\ERMag{\cT}} \Bigr] \leq (1+ \theta)\mathbb{E} \left[ \min_{m \in \cM} \lossb{\mathcal{A}_m(D_{n_t})} \right] + R_1(\theta)
\end{equation}
where $R_1(\theta) = R_{1,1}(\theta) + R_{1,2}(\theta)$ with
\begin{align*}
 R_{1,1}(\theta) &= \left( \theta + \frac{2 \bigl(1 + \log \lvert \cM \rvert \bigr)}{\theta} \right) \mathbb{E} \Bigl[ \delta^2 \bigl(\widehat{w}_{1,1}, \sqrt{n_v} \bigr) \Bigr] \enspace , \\
 R_{1,2}(\theta) &= \left( \theta + \frac{2 \bigl(1 + \log \lvert \cM \rvert \bigr) + \log^2 \lvert \cM \rvert}{\theta} \right) \mathbb{E}\Bigl[ \delta^2 \bigl(\widehat{w}_{1,2}, n_v \bigr) \Bigr] \enspace .
\end{align*}
Now, for any $D_{n_t}$-measurable functions $\widehat{w}_{2,1}$ and $\widehat{w}_{2,2}$ 
such that assumption $H (\widehat{w}_{2,1}, \widehat{w}_{2,2}, (\learnrule_m(D_{n_t}))_{m \in \cM}  )$ holds true almost surely, and any $x > 0$, $\theta \in (0;1]$, we have 
\begin{equation} \label{eq_2}
 (1-\theta) \mathbb{E} \Bigl[ \lossb{\ERMag{\cT}} \Bigr] \leq (1+ \theta)\mathbb{E} \left[ \min_{m \in \cM} \lossb{\mathcal{A}_m(D_{n_t})} \right] + R_2(\theta)
\end{equation}
where $R_2(\theta) = R_{2,1}(\theta) + R_{2,2}(\theta) + R_{2,3}(\theta) + R_{2,4}(\theta)$ with
\begin{align*} 
   R_{2,1}(\theta) &= \sqrt{2} \theta \mathbb{E} \left[ \delta^2 \left(\widehat{w}_{2,1}, \frac{\theta}{2}\sqrt{\frac{n_v}{x + \log \lvert \cM \rvert}} \right) \right]\enspace , \\
  R_{2,2}(\theta) &= \frac{\theta^2}{2} \mathbb{E} \left[ \delta^2 \left(\widehat{w}_{2,2},\frac{\theta^2}{4}\frac{n_v}{x + \log \lvert \cM \rvert} \right) \right]\enspace ,\\ 
  R_{2,3}(\theta) &= \mathrm{e}^{-x} R_{1,1}(\theta)\enspace , \\
\text{and} \qquad 
R_{2,4}(\theta)  &= \mathrm{e}^{-x} R_{1,2}(\theta) \enspace . 
\end{align*}
\end{Theorem}

\subsection{Proof of Theorem~\ref{hold_out_gen2}}

We start by proving three lemmas. 
 \begin{lemma} \label{fix_w2}
 Let $w$ be a non-decreasing function
 on $\mathbb{R}_+$. Let $r > 0$. 
Then 
  \[ \forall u \geq 0, w(u) \leq r \bigl(u^2 \vee \delta^2(w,r) \bigr) \enspace , \]
  where $\delta(w,r)$ is given by Definition~\ref{def_delta}.
 \end{lemma}

\begin{proof}
If $u > \delta(w,r)$, by Definition~\ref{def_delta},
\[ w(u) \leq r u^2 . \]
If $u \leq \delta(w,r)$, since $w$ is non-decreasing,
for all $v > \delta(w,r)$,
\[ w(u) \leq w(v) \leq r v^2 . \]
By taking the infimum over $v$, we recover
$w(u) \leq r \delta(w,r)^2$.
\end{proof}

\begin{lemma} \label{fix_w1}
 Let $w$ be a nondecreasing function such that 
 $x \mapsto \frac{w(x)}{x}$ is nonincreasing over $(0;+\infty)$. Let $a \in \mathbb{R}_+$ and $b \in (0;+\infty)$.
 For any $\theta \in (0;1]$ and $u \geq 0$,
 \[ \frac{a}{b} w(\sqrt{u}) \leq \frac{\theta}{2} \bigl[ u + \delta^2(w,b) \bigr] + \frac{a^2 \delta^2(w,b)}{\theta} \enspace  .\]
\end{lemma}

\begin{proof}
Since $w$ is nondecreasing,
\begin{align*}
 w(\sqrt{u}) &\leq w(\sqrt{u + \delta^2(w,b)}) \\
 &= \sqrt{u + \delta^2(w,b)} \frac{w(\sqrt{u + \delta^2(w,b)})}{\sqrt{u + \delta^2(w,b)}}.
\end{align*}
Since $\frac{w(x)}{x}$ is nonincreasing and $\delta(w,b) > 0$,
\begin{align*}
 w(\sqrt{u}) &\leq \sqrt{u + \delta^2(w,b)} \frac{w(\delta(w,b))}{\delta(w,b)} \\
 &\leq \sqrt{u + \delta^2(w,b)} b \delta(w,b) \text{ by Definition~\ref{def_delta}. }
\end{align*}
Therefore, using the inequality
$ \sqrt{ab} \leq \frac{\theta}{2} a + \frac{b}{2\theta} $, valid for any $a > 0, b > 0$, 
\[ \frac{a}{b} w(\sqrt{u}) \leq \sqrt{a^2 (u + \delta(w,b)^2)\delta(w,b)^2} 
 \leq \frac{\theta}{2}(u + \delta(w,b)^2) + \frac{a^2 \delta(w,b)^2}{\theta}. \]
\end{proof}

\begin{lemma} \label{min_lemm}
   Let $n_v \in \mathbb{N}^{*}$. Let $\cM$ be a finite set and let $(t_m)_{m \in \cM} \in \parspace^{\cM}$.
   Assume that there exists $p \in [0; 1/ \lvert \cM \rvert )$ and a function $R: (0;1] \rightarrow \mathbb{R}_+$ such that for any $m,m'$ in $\cM$,   with probability greater than $1-p$,
   \[ \forall \theta \in (0;1], \qquad (P_{n_v} - P)[\cfun(t_m,\cdot) - \cfun(t_{m'},\cdot)] \leq \theta \loss{t_m} + \theta \loss{t_{m'}} + R(\theta) \enspace  .\]
   Then for 
   $\widehat{m} \in \argmin_{m \in \cM} P_{n_v} \cfun(t_m,\cdot)$,
   with probability greater than $1 - \lvert \cM \rvert p$,
   \[\forall \theta \in (0;1], \qquad (1 - \theta) \loss{t_{\widehat{m}}} \leq (1 + \theta) \min_{m \in \cM} \loss{t_m} + R(\theta) \enspace  .\]
  \end{lemma}
\begin{proof}
   Let  $ m_* \in \argmin_{m \in \cM} P\cfun(t_m,\cdot)$.
  Then for any $m \in \cM$, with probability greater than $1 - p$,
  \[\forall \theta \in (0;1], (P_{n_v} - P)[\cfun(t_{m_*},\cdot) - \cfun(t_{m},\cdot)] 
  \leq \theta \loss{t_{m_*}} + \theta \loss{t_{m}} + R(\theta). \]
  So by the union bound, with probability greater than $1 - \lvert \cM \rvert p$,
  \[\forall \theta \in (0;1], \forall m \in \cM, (P_{n_v} - P)[\cfun(t_{m_*},\cdot) - \cfun(t_m,\cdot)] 
  \leq \theta \loss{t_{m_*}} + \theta \loss{t_{m}} + R(\theta) .\]
   On that event, for all $\theta \in (0;1],$
  \begin{align*}
   P\cfun(t_{\widehat{m}},\cdot) &= P_{n_v} \cfun(t_{\widehat{m}},\cdot) + (P - P_{n_v})\cfun(t_{\widehat{m}},\cdot) \\
   &\leq P_{n_v} \cfun(t_{m_*},\cdot) + (P - P_{n_v})\cfun(t_{\widehat{m}},\cdot) \\
   &= P\cfun(t_{m_*},\cdot) + (P - P_{n_v})[\cfun(t_{\widehat{m}},\cdot) - \cfun(t_{m_*},\cdot)] \\
   &\leq P\cfun(t_{m_*},\cdot) + \theta \loss{t_{m_*}} + \theta \loss{t_{\widehat{m}}} + R(\theta). 
   \end{align*}
Substracting the Bayes risk $P\cfun(\bayes,\cdot)$ on both sides, we get
with probability greater than $1 - \lvert \cM \rvert p$, for all $\theta \in (0;1],$
\begin{align*}
   \loss{t_{\widehat{m}}} &\leq \loss{t_{m_*}} + \theta \loss{t_{m_*}} + \theta \loss{t_{\widehat{m}}} + R(\theta), \\
  \text{that is, } (1 - \theta) \loss{t_{\widehat{m}}} &\leq (1 + \theta) \min_{m \in \cM} \loss{t_m} + R(\theta).
  \end{align*}
\end{proof}

We now prove Theorem~\ref{hold_out_gen2}.
Let $(m,m') \in \cM^2$ be fixed. Let
  \begin{equation} \label{eq_star2}
   \begin{aligned}
    \sigma &:= w_1(\sqrt{\loss{t_m}}) + w_1(\sqrt{\loss{t_{m'}}}), \\
\text{and} \qquad    c &:=  w_2(\sqrt{\loss{t_m}}) + w_2(\sqrt{\loss{t_{m'}}}) \enspace . 
   \end{aligned} 
  \end{equation}
By hypothesis $H \bigl(w_1,w_2,(t_m)_{m \in \cM} \bigr)$,
\begin{equation} \label{bern_hyp2}
  \exists c_{m,m'} \text{ such that } \forall k \geq 2, 
P\left(\cfun(t_m,\cdot) - \cfun(t_{m'},\cdot) - c_{m,m'}\right)^k \leq k! \sigma^2 c^{k-2} \enspace .
\end{equation}
For all $y > 0$, let $\Omega_y(m,m')$ be the event on which
 \begin{equation} \label{bern_bound2}
  (P_{n_v} - P)\bigl[ \cfun(t_m,\cdot) - \cfun(t_{m'},\cdot) \bigr] \leq \sqrt{\frac{2y}{n_v}} \sigma + \frac{cy}{n_v} \enspace . 
 \end{equation}
By Bernstein's inequality, $\mathbb{P}\bigl( \Omega_y(m,m') \bigr) \geq 1 - \mathrm{e}^{-y}$.

Let $q = \frac{\theta}{2} \sqrt{\frac{n_v}{x + \log \lvert \cM \rvert}}$.
By Lemma~\ref{fix_w2} with $r = q$, 
\[ \sigma :=  w_1(\sqrt{\loss{t_m}}) + w_1(\sqrt{\loss{t_{m'}}}) \leq 
q \left( \loss{t_m} \vee \delta^2(w_1,q) + \loss{t_{m'}} \vee \delta^2(w_1,q) \right) .\]
  Set $y = x + \log \lvert \cM \rvert$ in \eqref{bern_bound2}.
  Then 
    \begin{align*}
        \sqrt{\frac{2y}{n_v}} \sigma &:= \sqrt{\frac{2(x + \log \lvert \cM \rvert)}{n_v}} \sigma \\
        &\leq \sqrt{\frac{2(x + \log \lvert \cM \rvert)}{n_v}}  \frac{\theta}{2} \sqrt{\frac{n_v}{x + \log \lvert \cM \rvert}} \left( \loss{t_m} \vee \delta^2(w_1,q)+ \loss{t_{m'}} \vee \delta^2(w_1,q) \right) \\
        &\leq \frac{\theta}{\sqrt{2}} \left( \loss{t_m} + \loss{t_{m'}} + 2 \delta^2 \left(w_1, \frac{\theta}{2} \sqrt{\frac{n_v}{x + \log \lvert \cM \rvert}}\right) \right). \numberthis \label{bern_12}
       \end{align*}
  As for the second term of \eqref{bern_bound2}, by Lemma~\ref{fix_w2} with $r = q^2$, we have
  \[ c :=  w_2(\sqrt{\loss{t_m}}) + w_2(\sqrt{\loss{t_{m'}}}) \leq 
q^2 \left( \loss{t_m} \vee \delta^2(w_2,q^2) + \loss{t_{m'}} \vee \delta^2(w_2,q^2) \right). \]
Recall that $q$ is shorthand for $\frac{\theta}{2} \sqrt{\frac{n_v}{x + \log \lvert \cM \rvert}}$. Therefore:
  \begingroup 
  \begin{align*} 
    c \frac{y}{n_v} &\leq \frac{x + \log \lvert \cM \rvert}{n_v}  \frac{\theta^2}{4} \frac{n_v}{x + \log \lvert \cM \rvert} \left( \loss{t_m} \vee \delta^2(w_2,q^2) + \loss{t_{m'}} \vee \delta^2(w_2,q^2)\right) \\ 
    &= \frac{\theta^2}{4} \left( \loss{t_m} \vee \delta^2(w_2,q^2) + \loss{t_{m'}} \vee \delta^2(w_2,q^2) \right) \\
    &\leq \frac{\theta^2}{4} \left(\loss{t_m} + \loss{t_{m'}} + 2 \delta^2 \left(w_2,\frac{\theta^2}{4} \frac{n_v}{x + \log \lvert \cM \rvert}\right) \right) . \numberthis \label{bern_22}
  \end{align*} \endgroup
  Since $\sqrt{\frac{1}{2}} + \frac{1}{4} \leq 1$ and $\theta \in (0;1]$, plugging \eqref{bern_12} and \eqref{bern_22} in \eqref{bern_bound2} yields, 
  on the event $\Omega_{x + \log \lvert \cM \rvert}(m,m')$, for all $\theta \in (0;1]$,
  \begin{align*}
   (P_{n_v} - P)[\cfun(t_m,\cdot) - \cfun(t_{m'},\cdot)] &\leq \theta \bigl( \loss{t_m} + \loss{t_{m'}} \bigr) + \sqrt{2} \theta \delta^2 \left(w_1, \frac{\theta}{2} \sqrt{\frac{n_v}{x + \log \lvert \cM \rvert}} \right) \\ 
   &+ \frac{\theta^2}{2} \delta^2 \left(w_2, \frac{\theta^2}{4} \frac{n_v}{x + \log \lvert \cM \rvert} \right) . \numberthis \label{eq_dev}
  \end{align*}
  Suppose now that $x \mapsto \frac{w_j(x)}{x}$ is nonincreasing for $j \in \{1;2\}$.
  Let $\theta \in [0;1]$. Let $y \geq 0$.
  By Lemma~\ref{fix_w1} with $a = \sqrt{2y}$ and $b = \sqrt{n_v}$,
  \begin{align} 
\notag 
\sqrt{\frac{2y}{n_v}} \sigma 
&= \sqrt{\frac{2y}{n_v}}\left( w_1(\sqrt{\loss{t_m}}) + w_1(\sqrt{\loss{t_{m'}}}) \right)   
\\
\label{eq_w12} 
&\leq \frac{\theta}{2} \loss{t_m} + \frac{\theta}{2} \loss{t_{m'}} + \delta^2(w_1, \sqrt{n_v}) \left[\theta + \frac{2y}{\theta} \right] \enspace . 
\end{align}
  By Lemma~\ref{fix_w1} with $a = y$ and $b = n_v$,
  \begingroup 
  \begin{align*}
  c\frac{y}{n_v} &= \frac{y}{n_v} \left( w_2(\sqrt{\loss{t_m}}) + w_2(\sqrt{\loss{t_{m'}}}) \right) \\
  &\leq  \frac{\theta}{2} \loss{t_m} + \frac{\theta}{2} \loss{t_{m'}} + \delta^2(w_2,n_v) \left[\theta + \frac{y^2}{\theta} \right]. \numberthis \label{eq_w12_2} 
  \end{align*} \endgroup
  Plugging \eqref{eq_w12} and \eqref{eq_w12_2} in \eqref{bern_bound2} yields,
  on the event $\Omega_y(m,m')$, for all $\theta \in (0;1]$,
  \begin{equation} 
  \begin{split}
&\qquad   (P_{n_v} - P)[\cfun(t_m,\cdot) - \cfun(t_{m'},\cdot)] 
   \\
&\leq \theta \loss{t_m} + \theta \loss{t_{m'}} + \delta^2(w_1, \sqrt{n_v}) \left[\theta + \frac{2y}{\theta} \right] 
  + \delta^2(w_2,n_v) \left[\theta + \frac{y^2}{\theta} \right] \enspace .
  \end{split}
\label{eq_dev2}  
\end{equation}
  
By \eqref{eq_dev},
Lemma~\ref{min_lemm} applies with $p = \frac{\mathrm{e}^{-x}}{\lvert \cM \rvert}$ and 
\[ R(\theta) = \sqrt{2} \theta  \delta^2 \left(w_1,\frac{\theta}{2} \sqrt{\frac{n_v}{x + \log \lvert \cM \rvert}} \right)  + \frac{\theta^2}{2} \delta^2 \left(w_2, \frac{\theta^2}{4} \frac{n_v}{x + \log \lvert \cM \rvert}  \right). \]
This yields \eqref{eq_plus2}.
By \eqref{eq_dev2},
Lemma~\ref{min_lemm} applies with $p = \mathrm{e}^{-y}$ and 
\[ R(\theta) = \theta \left[\delta_1^2 + \delta_2^2 \right] + \frac{1}{\theta} \left[2y\delta_1^2 + y^2\delta_2^2 \right]. \]
Setting $y = \log \lvert \cM \rvert + x$ yields \eqref{eq_2plus2}. 
\qed

\subsection{Proof of Theorem~\ref{agcv_mean}}
We start by proving two lemmas. 
\begin{lemma} \label{prob_to_exp}
 Let $f \in L^1(\mathbb{R}_+,\mathrm{e}^{-x}\mathrm{d}x)$ be a non-negative, non-decreasing function such that $\lim\limits_{x \to +\infty} f(x) = + \infty$. Let
 $X$ be a random variable such that 
 \[ \forall x \in \mathbb{R}_+, \mathbb{P} \bigl( X > f(x) \bigr) \leq \mathrm{e}^{-x} \enspace . \]
 Then \[ \mathbb{E}[X] \leq \int_0^{+ \infty} f(x) \mathrm{e}^{-x} \mathrm{d}x \enspace . \]
\end{lemma}

\begin{proof}
 Let $g \in L^1(\mathbb{R}_+,\mathrm{e}^{-x}\mathrm{d}x)$ be a non-decreasing, differentiable function such that $g \geq f$.
 Then
 \begin{align*}
  \mathbb{E}[X] &\leq \int_{0}^{+\infty} \mathbb{P}[X > t] \mathrm{d}t \\
  &= \int_0^{g(0)} \mathbb{P}[X > t] \mathrm{d}t + \int_{0}^{+\infty} \mathbb{P}[X > g(x)] g'(x) \mathrm{d}x \\
  &\leq g(0) + \int_{0}^{+\infty} \mathrm{e}^{-x} g'(x)\mathrm{d}x \text{\quad since } g \geq f \\
  &= g(0) + [\mathrm{e}^{-x}g(x)]_0^{\infty} +\int_{0}^{+\infty}\mathrm{e}^{-x} g(x)\mathrm{d}x \\
  &= \int_{0}^{+\infty}\mathrm{e}^{-x} g(x)\mathrm{d}x \enspace .
 \end{align*}
 It remains to show that $g$ can approximate $f$ in $L^1(\mathbb{I}_{x \geq 0} \mathrm{e}^{-x} \mathrm{d}x)$.
 Let $K$ be a nonnegative smooth function vanishing outside $[-1;1]$, normalized such that $\int K(t) \mathrm{d}t = 1$. Let $\varepsilon > 0$.
 Define
 \begin{align}
  f_{\varepsilon}(x) &= \frac{1}{\varepsilon} \int f(t) K\left( \frac{x+\varepsilon - t}{\varepsilon} \right) \mathrm{d}t \label{eq_approx1} \\
  &= \frac{1}{\varepsilon} \int f(x + \varepsilon - t) K \left( \frac{t}{\varepsilon} \right) \mathrm{d}t \label{eq_approx2} 
 \end{align}
By \eqref{eq_approx1}, $f_{\varepsilon}$ is smooth.
By \eqref{eq_approx2}, $f_{\varepsilon}$ is nondecreasing, moreover 
\begin{align*}
 f_{\varepsilon}(x) - f(x) &= 
 \frac{1}{\varepsilon} \int \bigl[f(x + \varepsilon - t) - f(x) \bigr] K \left( \frac{t}{\varepsilon} \right) \mathrm{d}t \text{ since} \int K = 1 \\
 &= \frac{1}{\varepsilon} \int_{- \varepsilon}^\varepsilon \bigl[f(x + \varepsilon - t) - f(x) \bigr] K \left( \frac{t}{\varepsilon} \right) \mathrm{d}t \text{ since } K(u) = 0 \text{ when } |u| \geq 1\\
 &\geq 0 \text{ since } f \text{ is nondecreasing and } K \geq 0 \enspace .
\end{align*}
Thus $f_{\varepsilon} \geq f$.
Finally, by Jensen's inequality and Fubini's theorem,
\begin{align*}
 \int |f_{\varepsilon}(x) - f(x)|\mathrm{e}^{-x} \mathrm{d}x &\leq
 \frac{1}{\varepsilon} \int_{- \varepsilon}^\varepsilon K \left( \frac{t}{\varepsilon} \right) \int |f(x+\varepsilon-t) - f(x)|\mathrm{e}^{-x} \mathrm{d}x \\
 &\leq \sup_{|\tau| \leq 2\varepsilon} \int |f(x+\tau) - f(x)|\mathrm{e}^{-x} \mathrm{d}x \enspace ,
\end{align*}
which converges to $0$ when $\varepsilon \rightarrow 0$ since $f \in L^1(\mathbb{R}_+,\mathrm{e}^{-x}\mathrm{d}x)$.
\end{proof}

 We use the following additional notation:
\begin{Definition} \label{def_g}
 Let $g$ be the function defined by 
 \[ \forall (\theta,y,p,q) \in (0;1] \times \mathbb{R}_+^3 , \qquad g(\theta,y,p,q) = \theta [p + q] 
  + \frac{1}{\theta} \left[ 2y p + y^2 q \right] \enspace .\]
\end{Definition} 
This function satisfies the following properties. 
\begin{lemma} \label{int_calc}
Let $g$ be the function given in Definition~\ref{def_g}.
 For any $\theta \in [0;1]$ and any $u > 0, p \geq 0, q \geq 0$,
 \begin{align*}
  \mathrm{e}^u \int_u^{+ \infty} g(\theta,y,p,q) \mathrm{e}^{-y} \mathrm{d}y 
  &= \left(\theta + \frac{2(1+u)}{\theta} \right) p + \left(\theta + \frac{2 + 2u + u^2}{\theta} \right)q \enspace . 
 \end{align*}
\end{lemma}

\begin{proof} of Lemma~\ref{int_calc}

Using the formulas 
\begin{align*}
  \int_u^{+\infty} \mathrm{e}^{-x} \mathrm{d}x &= \mathrm{e}^{-u}, \int_u^{+\infty} x \mathrm{e}^{-x} \mathrm{d}x = (1+u)\mathrm{e}^{-u}, \\ 
  \int_u^{+\infty} x^2 \mathrm{e}^{-x} \mathrm{d}x &= (u^2 + 2u + 2)\mathrm{e}^{-u} \enspace ,
\end{align*}
we get:
\begin{align*}
 \mathrm{e}^u \int_u^{+ \infty} g(\theta,y,p,q) \mathrm{e}^{-y} \mathrm{d}y &= \theta [p+q] + \frac{2}{\theta}(1+u)p + (u^2 + 2u +2) \frac{q}{\theta} \\
 &= \left(\theta + \frac{2(1+u)}{\theta} \right) p + \left(\theta + \frac{2 + 2u + u^2}{\theta} \right)q \enspace .
\end{align*}

\end{proof}
%

We can now proceed with the proof of Theorem~\ref{agcv_mean}. Let $\theta \in (0;1]$ be fixed. 
Let $(\ERMho{T})_{T \in \cT}$ be the individual hold out estimators,
so that $\ERMag{\cT} = \frac{1}{\lvert \cT \rvert} \sum_{T \in \cT} \ERMho{T}$.
By convexity of the risk functional $\risk$, we have 
\[ \risk(\ERMag{\cT}) \leq \frac{1}{\lvert \cT \rvert} \sum_{T \in \cT} \risk(\ERMho{T}) \enspace . \]
It follows by substracting $\risk(\bayes)$ that:
\[ \loss{\ERMag{\cT}} \leq \frac{1}{\lvert \cT \rvert} \sum_{T \in \cT} \loss{\ERMho{T}} \enspace . \]
Since the data are i.i.d, by assumption \eqref{hyp.T},
all $\ERMho{T}$ have the same distribution.
Let $T_1 = \{1,\ldots,n_t \}$, so that $D_n^{T_1} = D_{n_t}$. 
Taking expectations yields
\begin{equation} \label{risk_ineq}
 \mathbb{E}[\loss{\ERMag{\cT}}] \leq \mathbb{E}[\loss{\ERMho{T_1}}] 
 \enspace . 
\end{equation}
Since $H\left(\widehat{w}_{1,1}, \widehat{w}_{1,2}, (\learnrule_m(D_{n_t})_{m \in \cM}) \right)$ holds, we can apply Theorem~\ref{hold_out_gen2} conditionally on $D_{n_t}$, with $t_m = \learnrule_m(D_{n_t})$. 

\paragraph{Proof of \eqref{eq_12}}
For $i \in \{1;2\}$, let $\widehat{\delta}_{1,i} = \delta(\widehat{w}_{1,i},\sqrt{n_v}^i)$. Let $g$ be given in Definition~\ref{def_g}.
By Theorem~\ref{hold_out_gen2}, Equation~\eqref{eq_2plus2}, for any $z \geq 0$, with probability
greater than $1 - \mathrm{e}^{-z}$,
\begin{equation} \label{thm_ho_g}
 (1-\theta) \loss{\ERMho{T_1}} \leq (1+\theta) \min_{m \in \cM} \loss{t_{m}} + g\bigl(\theta, z + \log \lvert \cM \rvert,\widehat{\delta}^2_{1,1} , \widehat{\delta}^2_{1,2} \bigr) \enspace .
\end{equation}
As $g$ is nondecreasing in its second variable,
Lemma~\ref{prob_to_exp} applied to the random variable   
$(1-\theta) \loss{\ERMho{T_1}}$ yields:
\[ (1-\theta) \mathbb{E} \left[ \loss{\ERMho{T_1}} \bigl| D_n^{T_1} \right] \leq 
(1+\theta) \min_{m \in \cM} \loss{t_{m}}
+ \int_{\log \lvert \cM \rvert}^{+ \infty} g\bigl(\theta,y,\widehat{\delta}^2_{1,1},\widehat{\delta}^2_{1,2}  \bigr) \mathrm{e}^{-(y - \log \lvert \cM \rvert)} \mathrm{d}y\enspace . \]
Lemma~\ref{int_calc} yields 
\begin{align*}
 (1-\theta) \mathbb{E} \left[ \loss{\ERMho{T_1}} \bigl| D_n^{T_1} \right] &\leq 
(1+\theta) \min_{m \in \cM} \loss{t_{m}}
+ \left(\theta + \frac{2\left(1 + \log \lvert \cM \rvert \right)}{\theta} \right) \widehat{\delta}^2_{1,1} \\ 
&+ \left(\theta + \frac{2\left(1 + \log \lvert \cM \rvert \right) + \log^2 \lvert \cM \rvert}{\theta} \right) \widehat{\delta}^2_{1,2}\enspace .
\end{align*}
Taking expectations with respect to $D_n^{T_1} = D_{n_t}$,
\begin{align*}
 (1-\theta) \mathbb{E} \left[ \loss{\ERMho{T_1}} \right] &\leq  (1+\theta) \mathbb{E} \bigl[ \min_{m \in \cM} \loss{\learnrule_{m}(D_{n_t})} \bigr]
+ \left(\theta + \frac{2\left(1 + \log \lvert \cM \rvert \right)}{\theta} \right) \mathbb{E}\bigl[ \widehat{\delta}^2_{1,1} \bigr] \\ 
&+ \left(\theta + \frac{2\left(1 + \log \lvert \cM \rvert \right) + \log^2 \lvert \cM \rvert}{\theta} \right) \mathbb{E}\left[ \widehat{\delta}^2_{1,2} \right]\enspace .
\end{align*}
Equation~\eqref{eq_12} then follows from Equation~\eqref{risk_ineq}.

\paragraph{Proof of \eqref{eq_2}}
Fix $x \geq 0$.
For $i \in \{1;2\}$, let 
$\widehat{\delta}_{2,i} = \delta \left( \widehat{w}_{2,i},\left(\frac{\theta}{2}\sqrt{\frac{n_v}{x + \log \lvert \cM \rvert}} \right)^i \right). $

%
By Theorem~\ref{hold_out_gen2}, Equation~\eqref{eq_plus2}, with probability larger than $1 - \mathrm{e}^{-x}$, 
\begin{equation} \label{thm_ho_g2}
 (1- \theta) \loss{\ERMho{T_1}} \leq (1+\theta)\min_{m \in \cM} \loss{t_m} + \sqrt{2} \theta \widehat{\delta}^2_{2,1} + \frac{\theta^2}{2}\widehat{\delta}^2_{2,2} \enspace . 
\end{equation}
Combining \eqref{thm_ho_g} and \eqref{thm_ho_g2}, for any $z \geq 0$, with probability larger than $1 - \mathrm{e}^{-z}$,
\begin{align*}
 (1-\theta) \loss{\ERMho{T_1}} &\leq (1+\theta)\min_{m \in \cM} \loss{t_m} + \sqrt{2} \theta \widehat{\delta}^2_{2,1} + \frac{\theta^2}{2}\widehat{\delta}^2_{2,2} + \mathbb{I}_{z \geq x} g\bigl(\theta, z + \log \lvert \cM \rvert,\widehat{\delta}^2_{1,1} , \widehat{\delta}^2_{1,2} \bigr) \enspace .
\end{align*}
 By Lemma~\ref{prob_to_exp},
\begin{align*}
(1-\theta) \mathbb{E}\left[ \loss{\ERMho{T_1}} \bigl| D_n^{T_1} \right] &\leq (1+\theta)  \min_{m \in \cM} \loss{t_m} + \sqrt{2} \theta \widehat{\delta}^2_{2,1}+ \frac{\theta^2}{2}\widehat{\delta}^2_{2,2} \\ 
&\qquad + \int_{x + \log \lvert \cM \rvert}^{+ \infty} g\bigl(\theta, y,\widehat{\delta}^2_{1,1} , \widehat{\delta}^2_{1,2} \bigr)  \mathrm{e}^{-(y-\log \lvert \cM \rvert)} \mathrm{d}y \enspace .
\end{align*}
By Lemma~\ref{int_calc}, it follows that
\begin{align*} 
  (1-\theta) \mathbb{E}\left[ \loss{\ERMho{T_1}} \bigl| D_n^{T_1} \right] &\leq (1+ \theta)\min_{m \in \cM} \loss{t_m} + \sqrt{2} \theta \widehat{\delta}^2_{2,1} + \frac{\theta^2}{2} \widehat{\delta}^2_{2,2} \\ 
  &\qquad + \mathrm{e}^{-x} \left( \theta + \frac{2(1 + x + \log \lvert \cM \rvert)}{\theta} \right) \widehat{\delta}^2_{1,1}\\  
 &\qquad + \mathrm{e}^{-x} \left( \theta + \frac{2(1 + x + \log \lvert \cM \rvert) + (x +\log\lvert \cM \rvert)^2}{\theta} \right) \widehat{\delta}^2_{1,2} \enspace .
\end{align*}
Taking expectations with respect to $D_n^{T_1}$ and using inequality \eqref{risk_ineq} yields Equation~\eqref{eq_2} of Theorem~\ref{agcv_mean}.
\qed

\section{RKHS regression: proof of Theorem~\ref{rkhs_thm}} \label{sec_proof_rkhs}
In the following, for any $g: \mathbb{R} \times \mathbb{R} \rightarrow \mathbb{R}$ and $t: \cX \rightarrow \mathbb{R}$,
the function $(x,y) \mapsto g(t(x),y)$ is denoted by~$g \circ t$.

\subsection{Preliminary results}

Remark first that the RKHS norm dominates the supremum norm:
\begin{lemma} \label{rkhs_norm_bound}
 If $\kappa = \sup_{x} K(x,x) < + \infty$ then for any $t \in \mathcal{H}$,
 \[ \NormInfinity{t} \leq \sqrt{\kappa} \hNorm{t} \enspace .  \]
\end{lemma}
\begin{proof}
By definition of an RKHS, $\forall t \in \mathcal{H}, \forall x \in \cX, \langle t, K(x,\cdot) \rangle_{\mathcal{H}} = t(x)$. It follows that, for any $t \in \mathcal{H}$,
\begin{align*}
 \NormInfinity{t}^2 = \sup_x t(x)^2 &= \sup_x \langle t, K(x,\cdot) \rangle^2_\mathcal{H} \\
 &\leq \hNorm{t}^2 \sup_x \langle K(x,\cdot), K(x,\cdot) \rangle \\
&\leq \hNorm{t}^2 \sup_x K(x,x). \\
\end{align*}
\end{proof}
%
Using standard arguments, the following deviation inequality can be derived.
\begin{proposition} \label{erm_rkhs}
Let $\mathcal{H}$ denote a RKHS with bounded kernel $K: \cX \times \cX \rightarrow \mathbb{R}$. 
Let $\kappa = \sup_x K(x,x)$ and $h: \mathbb{R}^2 \rightarrow \mathbb{R}$ be Lipschitz in its first argument with Lipschitz constant $L$. For any $t \in \mathcal{H}$ and $r > 0$, denote 
\[ B_{\mathcal{H}}(t,r) = \bigl\{t' \in \mathcal{H} \,\vert\, \hNorm{t' - t} \leq r \bigr\} \enspace . \] 
Let $t_0 \in \mathcal{H}$. Then for any probability measure $P$ on $\cX \times \mathbb{R}$ and any $y > 0$,
\[ P^{\otimes n} \left[ \sup_{(t_1,t_2) \in B_{\mathcal{H}}(t_0,r)^2} (P_n - P) \bigl(h \circ t_1 - h \circ t_2 \bigr)   \geq 2(2 + \sqrt{2y})L\frac{r \sqrt{\kappa}}{\sqrt{n}}\right] \leq \mathrm{e}^{-y}\enspace . \]
\end{proposition}
\begin{proof}
Let $D_n = \left( X_i, Y_i \right)_{1 \leq i \leq n}$ be a dataset drawn from $P$.
Let $(\sigma_i)_{1 \leq i \leq n}$ be i.i.d Rademacher variables independent from $D_n$.
Denote by $R_n(\mathcal{F}) = \mathbb{E} \left[ \sup_{f \in \mathcal{F}} \frac{1}{n} \sum_{i = 1}^n \sigma_i f(X_i) \right]$ the Rademacher complexity of a class $\mathcal{F}$ of real valued functions.

By Lemma~\ref{rkhs_norm_bound}, for any $(t_1,t_2) \in B_{\mathcal{H}}(t_0,r)^2,$ 
\[ \NormInfinity{h \circ t_1 - h \circ t_2} \leq L \NormInfinity{t_1 - t_2} \leq L \left[ \NormInfinity{t_1 - t_0} + \NormInfinity{t_2 - t_0} \right] \leq 2L \sqrt{\kappa} r\enspace . \]
By symmetry under exchange of $t_1$ and $t_2$, notice that
\[R_n \left( \{ h \circ t_1 - h \circ t_2 | (t_1,t_2) \in B_{\mathcal{H}}(t_0,r)^2 \} \right) = \sup_{(t_1,t_2) \in B_{\mathcal{H}}(t_0,r)^2} \frac{1}{n} \left| \sum_{i = 1}^n \sigma_i (h \circ t_1 - h \circ t_2)(X_i) \right| \enspace .  \]
By the bounded difference inequality and \cite{Boucheron:2005}, Theorem 3.2, it follows that for any $y > 0$, with probability greater than $1 - \mathrm{e}^{-y}$,
\[\sup_{(t_1,t_2) \in B_{\mathcal{H}}(t_0,r)^2} (P_n-P)(h \circ t_1 - h \circ t_2) \leq 2 R_n \left( \{ h \circ t_1 - h \circ t_2 | (t_1,t_2) \in B_{\mathcal{H}}(t_0,r)^2 \} \right) + 2Lr\sqrt{\frac{2\kappa y}{n}}. \]
Moreover,
\begin{align*}
  & R_n \left( \{ h \circ t_1 - h \circ t_2 | (t_1,t_2) \in B_{\mathcal{H}}(t_0,r)^2 \} \right) \\
  &\leq R_n(\{h \circ t | t \in B_{\mathcal{H}}(t_0,r) \}) + R_n(\{- h \circ t | t \in B_{\mathcal{H}}(t_0,r) \}) \\ 
  &\leq 2L R_n(B_\mathcal{H}(t_0,r)) \text{by the contraction lemma (relevant version: \cite{Meir_Zhang:2003}, Theorem 7),} \\
  &= 2 L R_n(B_\mathcal{H}(0,r)) \text{ (by translation invariance).} 
\end{align*}
Finally, by  a classical computation (see for example \cite{Boucheron:2005}, Section 4.1.2),
\begin{align*}  
 & R_n \left( \{ h \circ t_1 - h \circ t_2 | (t_1,t_2) \in B_{\mathcal{H}}(t_0,r)^2 \} \right) \\
  &\leq 2L \frac{r}{n} \mathbb{E} \sqrt{\sum_{i = 1}^n K(X_i,X_i)} \\
  &\leq 2 Lr \sqrt{\frac{\kappa}{n}}\enspace .
\end{align*}
\end{proof}

The proof of Theorem~\ref{rkhs_thm} also uses the following peeling lemma.
\begin{lemma} \label{pealing_lemma}
Let $(Z_u)_{u \in T}$ be a stochastic process and $d: T \rightarrow \mathbb{R}_+$ be a function. Let $a \geq 0$ and $b \in (0;2]$ and assume that 
\begin{equation} \label{pealing_hyp}
 \forall r,y \geq 0, \mathbb{P} \left[ \sup_{u \in T :d(u) \leq r} Z_u \geq r\frac{1+ \sqrt{b(a+y)}}{\sqrt{n}} \right] \leq \mathrm{e}^{-y} \enspace .
\end{equation}
Then, for any $\theta \in (0;+\infty)$,
\[ \mathbb{P} \left[ \exists u \in T, Z_u \geq \theta d^2(u) + \frac{2 + b \bigl[ 1.1 + 2(a + y) \bigr]}{\theta n} \right] \leq \mathrm{e}^{-y}\enspace . \]
\end{lemma}

\begin{proof}
Let $x > 0$. Let $\eta \in (1;2]$, $j_m \in \mathbb{N}^*$ and $y_0 \in \mathbb{R}$ be absolute constants that will be determined later. Then
\begin{align*}
&\mathbb{I} \left\{ \sup_{u \in T} \frac{Z_u}{d^2(u) + x^2} \geq \frac{1+\sqrt{b(a+y)}}{x\sqrt{n}} \right\} \\ 
&\leq
\mathbb{I} \left\{ \sup_{u \in T: d(u) \leq x} \frac{Z_u}{d^2(u) + x^2}  \geq \frac{1+\sqrt{b(a+y)}}{x\sqrt{n}} \right\} \\
&+ \sum_{j=0}^{+\infty} \mathbb{I} \left\{ \sup_{u \in T: \eta^j x \leq d(u) \leq \eta^{j+1}x} \frac{Z_u}{d^2(u) + x^2} \geq \frac{1+\sqrt{b(a+y)}}{x\sqrt{n}} \right\} \\
&\leq \mathbb{I} \left\{ \sup_{u \in T: d(u) \leq x} \frac{Z_u}{x^2}  \geq \frac{1+\sqrt{b(a+y)}}{x\sqrt{n}} \right\} \\
&+ \sum_{j=0}^{+\infty} \mathbb{I} \left\{ \sup_{u \in T: \eta^j x \leq d(u) \leq \eta^{j+1}x} \frac{Z_u}{(1 + \eta^{2j}) x^2} \geq \frac{1+\sqrt{b(a+y)}}{x\sqrt{n}} \right\} \\
&\leq \mathbb{I} \left\{ \sup_{u \in T: d(u) \leq x} Z_u  \geq \frac{x(1+\sqrt{b(a+y)})}{\sqrt{n}} \right\} \\
&+ \sum_{j=0}^{+\infty} \mathbb{I} \left\{ \sup_{u \in T: d(u) \leq \eta^{j+1}x} Z_u \geq (1 + \eta^{2j})\frac{x(1+\sqrt{b(a+y)})}{\sqrt{n}} \right\} \enspace . \numberthis \label{pealing_indicatrix_bound} \\
\end{align*}
Notice that:
\begin{align*}
(1 + \eta^{2j})\frac{x(1+\sqrt{b(a+y)})}{\sqrt{n}} &= x\eta^{j+1}
\frac{\eta^{2j}+1}{\eta^{j+1}} \frac{1+\sqrt{b(a+y)}}{\sqrt{n}} \\
&= x\eta^{j+1} \frac{1+\sqrt{b(a+z_j)}}{\sqrt{n}}\enspace ,  
\end{align*}
where:
\begin{align*}
z_j &= \frac{1}{b} \left( \frac{\eta^{2j}+1}{\eta^{j+1}} - 1 + \frac{\eta^{2j}+1}{\eta^{j+1}} \sqrt{b(a+y)} \right)^2 -a \\
&\geq \frac{1}{b} \left[ \frac{\eta^{2j}+1}{\eta^{j+1}} - 1 \right]^2 + \left( \frac{\eta^{2j}+1}{\eta^{j+1}} \right)^2 y \qquad \text{since } a \geq 0 \text{ and } \eta^{2j} + 1 \geq \eta^{j+1}\enspace . 
\end{align*}
Taking expectations in \eqref{pealing_indicatrix_bound} and using hypothesis (\ref{pealing_hyp}), we obtain:
\[ \mathbb{P}\left[ \sup_{u \in T} \frac{Z_u}{d^2(u) + x^2} \geq \frac{1+\sqrt{b(a+y)}}{x\sqrt{n}} \right] \leq  \mathrm{e}^{-y} + \sum_{j = 0}^{+ \infty} \mathrm{e}^{-z_j}\enspace . \]
So for any $y \geq y_0$ ,
\begin{align*}
 &\mathbb{P}\left[ \sup_{u \in T} \frac{Z_u}{d^2(u) + x^2} \geq \frac{1+\sqrt{b(a+y)}}{x\sqrt{n}} \right] \\
\qquad &\leq  \mathrm{e}^{-y} + \mathrm{e}^{-y} \sum_{j = 0}^{+ \infty} \exp \left( -\frac{1}{b} \left[ \frac{\eta^{2j}+1}{\eta^{j+1}} - 1 \right]^2 - \left( \frac{(\eta^{2j}+1)^2}{(\eta^{j+1})^2} - 1 \right)y  \right) \\
&\leq  \mathrm{e}^{-y} + \mathrm{e}^{-y} \sum_{j = 0}^{+ \infty} \exp \left( -\frac{1}{b} \left[ \frac{\eta^{2j}+1}{\eta^{j+1}} - 1 \right]^2 - \left( \frac{(\eta^{2j}+1)^2}{(\eta^{j+1})^2} - 1 \right)y_0  \right) \enspace . \numberthis \label{ineq_peal_1}
\end{align*}
Now, we have
\begin{align*}
\exp \left( -\frac{1}{b} \left[ \frac{\eta^{2j}+1}{\eta^{j+1}} - 1 \right]^2 - \left( \frac{(\eta^{2j}+1)^2}{(\eta^{j+1})^2} - 1 \right)y_0  \right) &\leq \exp \left(- \left( \frac{(\eta^{2j}+1)^2}{(\eta^{j+1})^2} - 1 \right)y_0  \right) \\
&\leq \exp \left(y_0 - \eta^{2(j-1)}y_0 \right) \enspace . \numberthis \label{ineq_peal_2} 
\end{align*}
Let $u$ denote the sequence $u_j = \exp \left(y_0 - \eta^{2(j-1)}y_0 \right)$.
Then for $j \geq j_m$,
\begin{align*}
\log u_{j+1} - \log u_j &= \eta^{2(j-1)}y_0 - \eta^{2j}y_0 \\
&= y_0 (1 - \eta^2)\eta^{2(j-1)} \\
&\leq y_0 (1 - \eta^2) \eta^{2(j_m -1)} \text{ since } \eta > 1 \enspace . 
\end{align*}
Thus,
\[ \forall j \geq j_m, \ \ u_{j+1} \leq u_j \exp\left(-y_0 (\eta^2-1) \eta^{2(j_m -1)} \right)\enspace . \]
Therefore, we have 
\[ \forall j \geq 0, \ \ u_{j+j_m} \leq u_{j_m} \exp\left(-jy_0 (\eta^2-1) \eta^{2(j_m -1))} \right) \]
and
\[ \sum_{j = j_m}^{+ \infty } u_j \leq u_{j_m} \left[1 - \exp\left(-y_0 (\eta^2-1) \eta^{2(j_m -1)} \right) \right]^{-1} \enspace . \]
It follows from \eqref{ineq_peal_1} and \eqref{ineq_peal_2} that for any $y \geq y_0$, since $b \leq 2$,
\begin{align*}
 \mathrm{e}^y &\mathbb{P}\left[ \sup_u \frac{Z_u}{d^2(u) + x^2} \geq \frac{1+\sqrt{b(a+y)}}{x\sqrt{n}} \right] \\
\qquad &\leq 1 + \sum_{j = 0}^{j_m} \exp \left( -\frac{1}{2} \left[ \frac{\eta^{2j}+1}{\eta^{j+1}} - 1 \right]^2 - \left( \frac{(\eta^{2j}+1)^2}{(\eta^{j+1})^2} - 1 \right)y_0  \right) \\
&+ \frac{\exp \left(y_0 - \eta^{2(j_m-1)}y_0 \right)}{1 - \exp\left(-y_0 (\eta^2-1) \eta^{2(j_m -1)} \right)}\enspace . \numberthis \label{ineq_peal_3}
\end{align*}
On the other hand, when $y \leq y_0$, trivially, 
\[  \mathbb{P}\left[ \sup_u \frac{Z_u}{d^2(u) + x^2} \geq \frac{1+\sqrt{b(a+y)}}{x\sqrt{n}} \right] \leq 1 \leq \mathrm{e}^{y_0} \mathrm{e}^{-y}. \]
Taking $\eta = 1.18, j_m = 10, y_0 = 0.52$, the right-hand side of \eqref{ineq_peal_3} evaluates to $1.6765 < 1.7$ whereas $\mathrm{e}^{y_0} \leq 1.683 < 1.7$. It follows that for all $y > 0$,
\begin{equation} \label{ineq_peal_4}
 \mathbb{P}\left[ \sup_u \frac{Z_u}{d^2(u) + x^2} \geq \frac{1+\sqrt{b(a+y)}}{x\sqrt{n}} \right] \leq 1.7 \mathrm{e}^{-y}\enspace .
\end{equation}

Now take $x = \frac{1 + \sqrt{b(a+y)}}{\theta \sqrt{n}}$ with $\theta > 0$.
We can rewrite:
\begin{align*}
\mathbb{P}\left[ \sup_u \frac{Z_u}{d^2(u) + x^2} \geq \frac{1 + \sqrt{b(a+y)}}{x\sqrt{n}} \right] &= \mathbb{P} \left[ \exists u \in T, \frac{Z_u}{d^2(u) + x^2} \geq \theta \right] \\
&= \mathbb{P} \left[ \exists u \in T, Z_u \geq \theta d^2(u) + \frac{1}{\theta n} \left(1 + \sqrt{b(a+y)}\right)^2 \right] \\
&\geq \mathbb{P} \left[ \exists u \in T, Z_u \geq \theta d^2(u) + \frac{2 + 2b(a+y)}{\theta n}\right]\enspace .
\end{align*}
It follows from Equation~\eqref{ineq_peal_4}, with $y$ replaced by $y+0.55$, that
\begin{align*}
 \mathbb{P} \left[ \exists u \in T, Z_u \geq \theta d^2(u) + \frac{2 + b(1.1 + 2(a + y))}{\theta n}  \right] &\leq 1.7 \mathrm{e}^{-0.55} \mathrm{e}^{-y} \\
 &\leq \mathrm{e}^{-y}\enspace .
\end{align*}
\end{proof}

We need two other technical lemmas in the proof of Theorem~\ref{rkhs_thm}.
 
\begin{lemma} \label{ridge_reg_prop}
For any nonnegative, continuous convex function $h$ over a Hilbert space $\mathcal{H}$, and any $\lambda \in \mathbb{R}_+$,
the elements of the regularization path,
\[t_\lambda = \argmin_{t \in \mathcal{H}} \left\{ h(t) + \lambda \hNorm{t}^2 \right\}\enspace ,\] satisfy, for any $(\lambda, \mu) \in \mathbb{R}^2$ such that $0 < \lambda \leq \mu$,
\[ \hNorm{t_\lambda - t_\mu}^2 \leq \hNorm{t_\lambda}^2 - \hNorm{t_\mu}^2\enspace .\]
\end{lemma}
\begin{proof}
By \cite[Theorem 2.11]{barbu2012convexity}, $t_\lambda$ exists for any $\lambda \in \mathbb{R}_{+}$
. Moreover, it is unique by strong convexity of $\Norm{\cdot}_{\mathcal{H}}^2$. 
For a closed convex set $\mathcal{C} \subset \mathcal{H}$, let $\Pi_{\mathcal{C}}$ denote the orthogonal projection onto $\mathcal{C}$.

Let $\mu > 0$. The set $\{ t : h(t) \leq h(t_\mu) \}$ is closed by continuity of $h$ and convex by convexity of $h$.
Moreover, for any $t \in \mathcal{H}$ such that $h(t) \leq h(t_\mu)$,
\begin{align*}
 \mu \hNorm{t_\mu}^2 &\leq h(t_\mu) - h(t) + \mu \hNorm{t_\mu}^2 \\
 &\leq \mu \hNorm{t}^2 \text{ by definition of } t_\mu\enspace .
\end{align*}
Therefore, $t_\mu = \Pi_{\{ t : h(t) \leq h(t_\mu) \}}(0)$.
Let $\lambda \in (0;\mu)$. 
By definition of $t_\lambda, t_\mu$,
\begin{align*}
 \frac{h(t_\mu)}{\mu} + \hNorm{t_\mu}^2 &\leq \frac{h(t_\lambda)}{\mu} + \hNorm{t_\lambda}^2 \\
 &= \frac{h(t_\lambda)}{\lambda} + \hNorm{t_\lambda}^2 + \left(\frac{1}{\mu} - \frac{1}{\lambda} \right) h(t_\lambda) \\
 &\leq \frac{h(t_\mu)}{\lambda} + \hNorm{t_\mu}^2 + \left(\frac{1}{\mu} - \frac{1}{\lambda} \right) h(t_\lambda)\enspace ,
\end{align*}
which implies $(\mu^{-1} - \lambda^{-1}) h(t_\mu) \leq (\mu^{-1} - \lambda^{-1}) h(t_\lambda)$ and thus $h(t_\lambda) \leq h(t_\mu)$ since $\lambda < \mu$.
For a projection $\Pi_{\mathcal{C}}$, it is well known that:
\[ \forall t \in \mathcal{H}, \forall t' \in \mathcal{C}, \langle t - \Pi_{\mathcal{C}}(t), \Pi_{\mathcal{C}}(t) - t' \rangle_{\mathcal{H}} \geq 0\enspace . \]
Choosing $\mathcal{C} = \{ t : h(t) \leq h(t_\mu) \}, t' = t_\lambda \in \mathcal{C}, t = 0$  yields $\langle - t_\mu, t_\mu - t_\lambda \rangle_{\mathcal{H}} \geq 0.$
Therefore 
\begin{align*}
 \hNorm{t_\lambda}^2 &= \hNorm{t_\mu + (t_\lambda - t_\mu)}^2 \\
 &= \hNorm{t_\mu}^2 + \hNorm{t_\lambda - t_\mu}^2 + 2 \langle t_\mu ,t_\lambda - t_\mu \rangle_{\mathcal{H}} \\
 &\geq \hNorm{t_\mu}^2 + \hNorm{t_\lambda - t_\mu}^2\enspace .
\end{align*}
\end{proof}

\begin{lemma} \label{fixed_pt}
Let $(b,c) \in \mathbb{R}_+^2$
and $l_{b,c}(x) = bx + c$. Let $\delta$ be given by Definition~\ref{def_delta}. For any $r \in \mathbb{R}_+$,
\begin{equation} \label{lin_fp}
 \delta^2(l_{b,c},r) \leq \frac{b^2}{r^2} + \frac{2c}{r}\enspace .
\end{equation}
For $(a,b,c) \in \mathbb{R}_+^3$, let
$g_{a,b,c}(x) = ax \vee \left[bx^3 + cx^2 \right]^{\frac{1}{2}}$.
For any $r \in \mathbb{R}_+$,
\begin{equation} \label{rkhs_fp}
 \delta^2 \bigl( g_{a,b,c},r \bigr) \leq \frac{a^2}{r^2} \vee \left[ \frac{b^2}{r^4} + \frac{2c}{r^2} \right] \leq \frac{a^2}{r^2} + \frac{b^2}{r^4} + \frac{2c}{r^2}\enspace . 
\end{equation}
\end{lemma}

\begin{proof}
Since $x \mapsto \frac{l_{b,c}(x)}{x}$ is nonincreasing, we have by Remark~\ref{rmk_delta}:
\begin{align*}
b \delta(l_{b,c},r) + c = r \delta^2(l_{b,c},r), \text{ i.e} \\
\delta^2(l_{b,c},r) - \frac{b \delta(l_{b,c},r)}{r} - \frac{c}{r} &= 0\enspace .
\end{align*}
Hence $\delta(l_{b,c},r) = \frac{b}{2 r} + \frac{1}{2} \sqrt{\frac{b^2}{r^2} + \frac{4c}{r}}$. Thus
\[
\delta^2(l_{b,c},r) \leq 2 \left( \frac{b^2}{4r^2} + \frac{b^2}{4r^2} + \frac{c}{r} \right) \leq \frac{b^2}{r^2} + \frac{2c}{r}.
\]
This proves \eqref{lin_fp}. 
For any $x > 0$, $g_{a,b,c}(x) \leq rx^2$ is equivalent to
\begin{align}
 ax &\leq rx^2 \label{fixed_pt_cond_1}\\
 \text{ and } bx^3 + cx^2 &\leq r^2 x^4 \enspace .\label{fixed_pt_cond_2}
\end{align}
Eq.~\eqref{fixed_pt_cond_1} is equivalent to $x \geq \frac{a}{r}$. On the other hand, 
\begin{align*}
x > \left[\frac{b^2}{r^4} + \frac{2c}{r^2} \right]^{\frac{1}{2}} &\implies x > \delta(l_{b,c},r^2) \text{ by } \eqref{lin_fp} \\
&\implies bx + c \leq r^2 x^2 \text{ by Definition~\ref{def_delta} }  \\
&\implies \eqref{fixed_pt_cond_2}.
 \end{align*}
Therefore, whenever
\[ x > \frac{a}{r} \vee \left[\frac{b^2}{r^4} + \frac{2c}{r^2} \right]^{\frac{1}{2}}\enspace ,\]
it holds that $g_{a,b,c}(x) \leq rx^2$.
\eqref{rkhs_fp} follows by Definition~\ref{def_delta}.
\end{proof}

\subsection{Uniform control on the empirical process}
From now on until the end of the proof, the notation and hypotheses of Theorem~\ref{rkhs_thm} are used. 
Recall also the notation $g \circ t: (x,y) \mapsto g(t(x),y)$, for any $g: \mathbb{R} \times \mathbb{R} \rightarrow \mathbb{R}$ and $t: \cX \rightarrow \mathbb{R}$. 
Fix a training set $D_{n_t}$.
Start with the following definition.
\begin{Definition} \label{def_d_y}
 For $t_1,t_2 \in \mathcal{H}$, let
 \begin{equation}
  d(t_1,t_2) = \min_{\lambda \in \Lambda} \hNorm{t_1 - \bayes_\lambda} + \hNorm{t_1 - t_2} \enspace , \label{def_d}
 \end{equation}
where $\bayes_\lambda = \argmin_{t \in \mathcal{H}} \left\{ P (\ctrain\circ t ) + \lambda \hNorm{t}^2 \right\} $.
Furthermore, let
\[ \widehat{y} = \frac{\lambda_m n_t}{32\kappa L^2} \times \sup_{(t_1,t_2) \in \mathcal{H}^2} \left\{(P_{n_t} - P) (\ctrain\circ t_1 - \ctrain\circ t_2) - \frac{\lambda_m}{2} d(t_1,t_2)^2 \right\}\enspace , \]
so that
\begin{equation} \label{em_proc_ub}
 \forall (t_1,t_2) \in \mathcal{H}^2, (P_{n_t} - P) (\ctrain\circ t_1 - \ctrain\circ t_2) \leq \frac{\lambda_m}{2} d(t_1,t_2)^2 + \frac{32\kappa L^2 \widehat{y}}{\lambda_m n_t}\enspace .
\end{equation}
\end{Definition}
We then have  the following bounds on $\widehat{y}$.
\begin{claim} \label{y_bound}
 For all $x \geq 0$,
 \[ \mathbb{P} \bigl( \widehat{y} \geq 2.6 + \log \lvert \Lambda \rvert + x \bigr) \leq \mathrm{e}^{-x}\enspace .\]
 In particular, $\mathbb{E}[\widehat{y} ] \leq 4 + \log \lvert \Lambda \rvert$.
\end{claim}
\begin{proof}
Let $(t_1,t_2) \in \mathcal{H}$ be such that $d(t_1,t_2) \leq r$. 
Let $\lambda \in \Lambda$ be such that $\hNorm{t_1 - \bayes_\lambda} + \hNorm{t_1 - t_2} \leq r$.
By the triangle inequality, $t_1,t_2 \in B(\bayes_\lambda, r).$
Hence
\begin{equation} \label{eq_d_bd}
 \sup_{(t_1,t_2): d(t_1,t_2) \leq r} \left\{ (P_{n_t} - P)(\ctrain\circ t_1 - \ctrain\circ t_2) \right\} \leq \max_{\lambda \in \Lambda} \sup_{(t_1,t_2) \in B(\bayes_\lambda, r)^2} (P_{n_t} - P) (\ctrain\circ t_1 - \ctrain\circ t_2).
\end{equation}
From Proposition~\ref{erm_rkhs} and the union bound, it follows that, for any $x \geq 0$,
\[ \mathbb{P} \left[ \max_{\lambda \in \Lambda} \sup_{(t_1,t_2) \in B(\bayes_\lambda, r)^2} (P_{n_t} - P) (\ctrain\circ t_1 - \ctrain\circ t_2) \geq 2\left( 2 + \sqrt{2(x+\log \lvert \Lambda \rvert)} \right) L \frac{r\sqrt{\kappa}}{\sqrt{n_t}} \right] \leq \mathrm{e}^{-x}. \]
It follows by  Equation~\eqref{eq_d_bd} that, for all $x \geq 0$,
\[ \mathbb{P} \left[ \sup_{(t_1,t_2) : d(t_1,t_2) \leq r} \frac{1}{4L\sqrt{\kappa}} (P_{n_t} - P)(\ctrain\circ t_1 - \ctrain\circ t_2) \geq \left( 1 + \sqrt{\frac{x+\log \lvert \Lambda \rvert}{2}} \right) \frac{r}{\sqrt{n_t}} \right] \leq \mathrm{e}^{-x}.\]
By Lemma~\ref{pealing_lemma} with $\theta = \frac{\lambda_m}{8L\sqrt{\kappa}}$, $a = \log \lvert \Lambda \rvert$, $b =\frac{1}{2}$, with probability larger than $1 - \mathrm{e}^{-x}$,
\[ \forall (t_1,t_2), (P_{n_t} - P)(\ctrain\circ t_1 - \ctrain\circ t_2) \leq \frac{\lambda_m}{2} d(t_1,t_2)^2 + 32L^2 \frac{\kappa(2.6 + x + \log \lvert \Lambda \rvert)}{\lambda_m n_t}. \]
On the same event, $\widehat{y} \leq 2.6 + x + \log \lvert \Lambda \rvert$ by Definition~\ref{def_d_y}.

Therefore, by Lemma~\ref{prob_to_exp}, $\mathbb{E}[\widehat{y}] \leq 3.6 + \log \lvert \Lambda \rvert$.
\end{proof}

Definition~\ref{def_d_y} and Proposition~\ref{y_bound} together imply a uniform control on the empirical process thanks to the drift term $\lambda_m d(t_1,t_2)^2$, whereas Proposition~\ref{prop_unif_bound} only gave a bound on an RKHS ball of fixed radius.

\subsection{Verifying the assumptions of Theorem~\ref{agcv_mean}} 
Theorem~\ref{rkhs_thm} is a consequence of Theorem~\ref{agcv_mean}. For all $\lambda \in \Lambda$, let $\widehat{t}_\lambda = \learnrule_\lambda(D_{n_t})$, where $\learnrule_\lambda$ is given by Definition~\ref{def_kern}. To verify the assumptions of Theorem~\ref{agcv_mean}, adequate functions $(\widehat{w}_{i,j})_{(i,j) \in \{1;2\}^2}$ must be found such that for $i \in \{ 1;2\}$, $H \left(\widehat{w}_{i,1},\widehat{w}_{i,2}, (\widehat{t}_\lambda)_{\lambda \in \Lambda} \right)$ holds almost surely . This is the purpose of this section.

The core of the proof of Theorem~\ref{rkhs_thm} lies in the following deterministic claim.

\begin{claim} \label{prop_unif_bound}
For all $\lambda, \mu \in \Lambda$ such that $\lambda \leq \mu$,
\[ \NormInfinity{\widehat{t}_\lambda - \widehat{t}_\mu}^2 \leq \frac{\kappa C}{\lambda_m} \loss{\widehat{t}_\mu} + 96L^2 \frac{\kappa^2 \widehat{y}}{\lambda_m^2 n_t}\enspace .  \]
\end{claim}
\begin{proof}
Let $(\lambda, \mu) \in \Lambda^2$ with $\lambda \leq \mu$.
Let $\bayes_\mu$ be as in Definition~\ref{def_d_y}, Equation~\eqref{def_d}. By convexity of $\ctrain$, the function $t \mapsto P(\ctrain\circ t) + \mu \hNorm{t}^2$ is $\mu$-strongly convex. Since $\bayes_\mu$ is its optimum, we get 
\[ \forall t \in \mathcal{H}, P(\ctrain\circ t) + \mu \hNorm{t}^2 \geq P(\ctrain\circ \bayes_\mu) + \mu \hNorm{\bayes_\mu}^2 + \mu \hNorm{t - \bayes_\mu}^2 \enspace . \]
Hence, taking $t = \widehat{t}_\mu$,
\begin{align*}
\lambda_m \hNorm{\widehat{t}_\mu - \bayes_\mu}^2 &\leq \mu \hNorm{\widehat{t}_\mu - \bayes_\mu}^2 \\ 
&\leq P(\ctrain\circ \widehat{t}_\mu) + \mu \hNorm{\widehat{t}_\mu}^2 - P (\ctrain\circ \bayes_\mu) - \mu \hNorm{\bayes_\mu}^2 \\
&= P_{n_t}(\ctrain\circ \widehat{t}_\mu) + \mu \hNorm{\widehat{t}_\mu}^2 - P_{n_t} (\ctrain\circ \bayes_\mu) - \mu \hNorm{\bayes_\mu}^2 + (P - P_{n_t})(\ctrain\circ \widehat{t}_\mu - \ctrain\circ \bayes_\mu) \enspace . \end{align*}
By Definition~\ref{def_kern},
\[ P_{n_t}(\ctrain\circ \widehat{t}_\mu) + \mu \hNorm{\widehat{t}_\mu}^2 \leq P_{n_t} (\ctrain\circ \bayes_\mu) + \mu \hNorm{\bayes_\mu}^2. \]
Hence $\lambda_m \hNorm{\widehat{t}_\mu - \bayes_\mu}^2 \leq (P - P_{n_t})(\ctrain\circ \widehat{t}_\mu - \ctrain\circ \bayes_\mu) = (P_{n_t} - P)(\ctrain\circ \bayes_\mu - \ctrain\circ \widehat{t}_\mu)$.
Now take $t_1 = \bayes_\mu$ and $t_2 = \widehat{t}_\mu$ in Equation~\eqref{em_proc_ub} of Definition~\ref{def_d_y} to get
\begin{align*}
\lambda_m \hNorm{\widehat{t}_\mu - \bayes_\mu}^2
&\leq \frac{\lambda_m}{2} d(\bayes_\mu, \widehat{t}_\mu)^2 + 32L^2 \frac{\kappa \widehat{y}}{\lambda_m n_t} \\
&= \frac{\lambda_m}{2} \hNorm{\bayes_\mu - \widehat{t}_\mu}^2 + 32L^2 \frac{\kappa \widehat{y}}{\lambda_m n_t} \enspace .
\end{align*}
Therefore,
\begin{equation} \label{eq_risk_min}
 \hNorm{\widehat{t}_\mu - \bayes_\mu}^2 \leq 64 L^2 \frac{\widehat{y} \kappa}{\lambda_m ^2 n_t} \enspace .
\end{equation}
Now $\hNorm{\widehat{t}_\lambda - \widehat{t}_\mu}^2$ can be bounded as follows. 
Since $t \mapsto P_{n_t}(\ctrain\circ t) + \lambda \hNorm{t}^2$ is $\lambda$-strongly convex and $\widehat{t}_\lambda$ is its optimum,
\begin{align*}
\lambda_m \hNorm{\widehat{t}_\lambda - \widehat{t}_\mu}^2 &\leq \lambda \hNorm{\widehat{t}_\lambda - \widehat{t}_\mu}^2 \\
&\leq P_{n_t}( \ctrain\circ \widehat{t}_\mu) - P_{n_t} (\ctrain\circ \widehat{t}_\lambda) + \lambda \hNorm{\widehat{t}_\mu}^2 - \lambda \hNorm{\widehat{t}_\lambda}^2 \enspace . 
\end{align*}
By Lemma~\ref{ridge_reg_prop} with $h(t) = P_{n_t}(\ctrain\circ t)$, $\hNorm{\widehat{t}_\lambda - \widehat{t}_\mu}^2 \leq  \hNorm{\widehat{t}_\lambda}^2 - \hNorm{\widehat{t}_\mu}^2. $ Hence 
\begin{align*}
(\lambda_m + \lambda) \hNorm{\widehat{t}_\lambda - \widehat{t}_\mu}^2 &\leq P_{n_t}( \ctrain\circ \widehat{t}_\mu) - P_{n_t} (\ctrain\circ \widehat{t}_\lambda)  \\
&= P(\ctrain\circ \widehat{t}_\mu) - P(\ctrain\circ \widehat{t}_\lambda) + (P_{n_t} - P)\left[ \ctrain\circ \widehat{t}_\mu - \ctrain\circ \widehat{t}_\lambda  \right] \\
&\leq P(\ctrain\circ \widehat{t}_\mu) - \min_{t \in \parspace} P(\ctrain\circ t) + (P_{n_t} - P)\left[ \ctrain\circ \widehat{t}_\mu - \ctrain\circ \widehat{t}_\lambda \right] \\
&\leq C\loss{\widehat{t}_\mu} + (P_{n_t} - P)\left[ \ctrain\circ \widehat{t}_\mu - \ctrain\circ \widehat{t}_\lambda  \right] \text{ by hypothesis } Comp_C(\cpred,c) \enspace . 
\end{align*}
By Definition~\ref{def_d_y}, Equation~\eqref{em_proc_ub} with $t_1 = \widehat{t}_\mu$ and $t_2 = \widehat{t}_\lambda$,
\begin{align*}
(\lambda_m + \lambda) \hNorm{\widehat{t}_\lambda - \widehat{t}_\mu}^2 &\leq C\loss{\widehat{t}_\mu} + \frac{\lambda_m}{2} \left[ \hNorm{\widehat{t}_\mu - \bayes_\mu} + \hNorm{\widehat{t}_\lambda - \widehat{t}_\mu} \right]^2 + 32L^2 \frac{\kappa \widehat{y}}{\lambda_m n_t}  \\
&\leq C\loss{\widehat{t}_\mu} + \frac{\lambda_m}{2} \left[ 8 \frac{L \sqrt{\widehat{y} \kappa}}{\lambda_m \sqrt{n_t}} + \hNorm{\widehat{t}_\lambda - \widehat{t}_\mu} \right]^2 + 32L^2 \frac{\kappa \widehat{y}}{\lambda_m n_t} \text{ by equation } \eqref{eq_risk_min}. 
\end{align*}
For any $(a,b)$, $(a+b)^2 \leq 2a^2 + 2b^2$, hence
\begin{align*}
(\lambda + \lambda_m) \hNorm{\widehat{t}_\lambda - \widehat{t}_\mu}^2 
&\leq C\loss{\widehat{t}_\mu} + \frac{\lambda_m}{2} \left[ 128 L^2 \frac{\widehat{y} \kappa}{\lambda_m^2 n_t} + 2\hNorm{\widehat{t}_\lambda - \widehat{t}_\mu}^2 \right] + 32L^2 \frac{\kappa \widehat{y}}{\lambda_m n_t}.
\end{align*}

This yields:
\[ \lambda \hNorm{\widehat{t}_\lambda - \widehat{t}_\mu}^2 \leq C\loss{\widehat{t}_\mu} + 96 L^2 \frac{\kappa \widehat{y}}{\lambda_m n_t}, \]
and finally, since $\lambda \geq \lambda_m$:
\[\hNorm{\widehat{t}_\lambda - \widehat{t}_\mu}^2 \leq \frac{C\loss{\widehat{t}_\mu}}{\lambda_m} + 96 L^2 \frac{ \kappa \widehat{y}}{\lambda_m^2 n_t}. \]
Now, by Lemma~\ref{rkhs_norm_bound},
\begin{align*}
\NormInfinity{\widehat{t}_\lambda - \widehat{t}_\mu}^2
&\leq \kappa \hNorm{\widehat{t}_\lambda - \widehat{t}_\mu}^2  \\
&\leq \frac{\kappa C}{\lambda_m} \loss{\widehat{t}_\mu} + 96L^2 \frac{\kappa^2 \widehat{y}}{\lambda_m^2 n_t}.  
\end{align*}
This proves Claim~\ref{prop_unif_bound}.
\end{proof}

Using hypothesis $SC_{\cvxct,\slope}$ ---Equation~\eqref{hyp.SC}---, a refined bound can be obtained on $P \left[ \bigl(\cpred \circ \widehat{t}_\lambda - \cpred \circ \widehat{t}_\mu \bigr)^2 \right] $.
\begin{claim} \label{l2_bound}
For any $(\lambda, \mu) \in \Lambda^2$,
\[P \left[ \bigl(\cpred \circ \widehat{t}_\lambda - \cpred \circ \widehat{t}_\mu \bigr)^2 \right] \leq \widehat{w}_B \left( \sqrt{\loss{\widehat{t}_\lambda}} \right)^2 + \widehat{w}_B \left(\sqrt{\loss{\widehat{t}_\mu}} \right)^2\]
where 
\[\widehat{w}_B(x)^2 = \max \left\{ \cvxct x^2, \  \slope\frac{4}{3}\sqrt{\frac{\kappa C}{\lambda_m}} x^3 + 10 \slope L \frac{\kappa \sqrt{\widehat{y}}}{\lambda_m \sqrt{n_t}} x^2   \right\}\enspace . \]
\end{claim}
\begin{proof}
By hypothesis $SC_{\cvxct,\slope}$ ---Equation~\eqref{hyp.SC}--- with $u = \widehat{t}_\lambda (X)$ and $v = \widehat{t}_\mu (X)$,
\begin{align*}
 \mathbb{E} \left[(\cpred \circ \widehat{t}_\lambda - \cpred \circ \widehat{t}_\mu )^2 (X,Y) | X \right] &\leq \left[ \cvxct \vee \bigl(\slope |\widehat{t}_\lambda(X)- \widehat{t}_\mu (X)| \bigr) \right] \left[\xloss{\widehat{t}_\lambda(X)} + \xloss{\widehat{t}_\mu(X)} \right] \\
&\leq \left[ \cvxct \vee \bigl(\slope \NormInfinity{\widehat{t}_\lambda- \widehat{t}_\mu} \bigr) \right] \left[\xloss{\widehat{t}_\lambda(X)} + \xloss{\widehat{t}_\mu(X)} \right],
\end{align*}
where $\xloss{u} = \mathbb{E}[\cpred(u,Y)| X] - \min_{v \in \mathbb{R}} \mathbb{E}[\cpred(v,Y)| X] $.
Integrating this inequality with respect to $X$, it follows that, 
\begin{align*}
P \left[ \bigl(\cpred \circ \widehat{t}_\lambda - \cpred \circ \widehat{t}_\mu \bigr)^2 \right] &\leq \bigl[ \cvxct \vee \bigl(\slope \NormInfinity{\widehat{t}_\lambda- \widehat{t}_\mu} \bigr) \bigr]\Bigl[ \loss{\widehat{t}_\lambda} + \loss{\widehat{t}_\mu } \Bigr].
\end{align*}
Assume without loss of generality that $\lambda \leq \mu$.
By Claim~\ref{prop_unif_bound},
\begin{align*}
P \left[ \bigl(\cpred \circ \widehat{t}_\lambda  - \cpred \circ \widehat{t}_\mu \bigr)^2 \right]
&\leq \left( \cvxct \vee \slope \left[ \sqrt{\frac{\kappa C}{\lambda_m}} \sqrt{\loss{\widehat{t}_\mu}} + 10\frac{L\kappa \sqrt{\widehat{y}}}{\lambda_m \sqrt{n_t}} \right] \right)\Bigl[\loss{\widehat{t}_\lambda} + \loss{\widehat{t}_\mu} \Bigr] \\
&\leq \max \Biggl\{ \cvxct \Bigl[ \loss{\widehat{t}_\lambda} + \loss{\widehat{t}_\mu} \Bigr],  \slope  \Biggl[ \sqrt{\frac{\kappa C}{\lambda_m}} \left( \sqrt{\loss{\widehat{t}_\mu}} \loss{\widehat{t}_\lambda} + \sqrt{\loss{\widehat{t}_\mu}}^3 \right) \\ 
&+ 10 \frac{L \kappa \sqrt{\widehat{y}}}{\lambda_m \sqrt{n_t}} \Bigl[\loss{\widehat{t}_\lambda} + \loss{\widehat{t}_\mu} \Bigr] \Biggr] \Biggr\}. \numberthis \label{pre_l2_bound}  
\end{align*}
Using the inequality $ab \leq \frac{a^p}{p} + \frac{b^q}{q}$ with Hölder conjugates $p = 3$, $q = \frac{3}{2}$, we have:
\begin{align*}
 \sqrt{\loss{\widehat{t}_\mu}} \loss{\widehat{t}_\lambda} + \sqrt{\loss{\widehat{t}_\mu}}^3 &\leq \frac{1}{3} \sqrt{\loss{\widehat{t}_\mu}}^3 + \frac{2}{3} \loss{\widehat{t}_\lambda}^\frac{3}{2} + \sqrt{\loss{\widehat{t}_\mu}}^3\\
&\leq \frac{4}{3} \left[  \sqrt{\loss{\widehat{t}_\lambda}}^3 + \sqrt{\loss{\widehat{t}_\mu}}^3 \right]. \numberthis \label{holder_ineq}
\end{align*}
Claim~\ref{l2_bound} then follows from inequalities \eqref{pre_l2_bound} and \eqref{holder_ineq} using the elementary inequality $(a+b) \vee (c+d) \leq a\vee c + b \vee d$.
\end{proof}

As $\cpred$ is $L$-Lipschitz in its first argument, it follows from Claim~\ref{prop_unif_bound}  that for all $\lambda, \mu \in \Lambda$ s.t. $\lambda \leq \mu$, 
\begin{align*}
 \NormInfinity{\cpred \circ \widehat{t}_\lambda - \cpred \circ \widehat{t}_\mu} &\leq L \NormInfinity{\widehat{t}_\lambda - \widehat{t}_\mu} \\ 
 &\leq L\sqrt{\frac{\kappa C}{\lambda_m}} \sqrt{\loss{\widehat{t}_\mu}} + 10 L^2\frac{\kappa \sqrt{\widehat{y}}}{\lambda_m \sqrt{n_t}} \\
&\leq \widehat{w}_A \left( \sqrt{\loss{\widehat{t}_\mu}} \right) + \widehat{w}_A \left(\sqrt{\loss{\widehat{t}_\lambda}} \right) \numberthis \label{gamma_unif_bound} \enspace ,
\end{align*}
where 
\begin{equation} \label{w2_def}
\widehat{w}_A(x) = L \sqrt{\frac{\kappa C}{\lambda_m}} x + 5L^2 \frac{\kappa \sqrt{\widehat{y}}}{\lambda_m \sqrt{n_t}}.
\end{equation}
If follows that
for all $k \geq 2$,
\begin{align*}
P \left[ \bigl(\cpred \circ \widehat{t}_\lambda - \cpred \circ \widehat{t}_\mu \bigr)^k \right] &\leq \NormInfinity{\cpred \circ \widehat{t}_\lambda - \cpred \circ \widehat{t}_\mu}^k \\ 
&\leq \left[ \widehat{w}_A \left(\sqrt{\loss{\widehat{t}_\mu}} \right) + \widehat{w}_A \left(\sqrt{\loss{\widehat{t}_\lambda}} \right) \right]^k \enspace .
\end{align*}
This proves that hypothesis $H\left(\widehat{w}_A,\widehat{w}_A, (\widehat{t}_\lambda)_{\lambda \in \Lambda} \right)$, as defined in Appendix~\ref{app.general-thms}, holds true.

It follows from Claim~\ref{l2_bound} and Equation~\eqref{gamma_unif_bound}
that, for all $k \geq 2$,
\begin{align*}
P \bigl[ |\cpred \circ \widehat{t}_\lambda - \cpred \circ \widehat{t}_\mu|^k \bigr]
&\leq \NormInfinity{\cpred \circ \widehat{t}_\lambda - \cpred \circ \widehat{t}_\mu}^{k-2}
P \left[ \bigl(\cpred(\widehat{t}_\lambda(X),Y) - \cpred(\widehat{t}_\mu(X),Y) \bigr)^2  \right] \\
&\leq \left[ \widehat{w}_A \Bigl(\sqrt{\loss{\widehat{t}_\lambda}} \Bigr) + \widehat{w}_A \Bigl(\sqrt{\loss{\widehat{t}_\mu}} \Bigr) \right]^{k-2} \\ 
& \qquad \times \left[\widehat{w}_B \Bigl( \sqrt{\loss{\widehat{t}_\lambda}} \Bigr) + \widehat{w}_B \Bigl(\sqrt{\loss{\widehat{t}_\mu}} \Bigr) \right]^2 \enspace ;
\end{align*}
which proves that $H \left( \widehat{w}_B,\widehat{w}_A,(\widehat{t}_\lambda)_{\lambda \in \Lambda} \right)$ holds true.
\subsection{Conclusion of the proof}
We have proved that $H \left( \widehat{w}_B,\widehat{w}_A,(\widehat{t}_\lambda)_{\lambda \in \Lambda} \right)$ and $H \left( \widehat{w}_A,\widehat{w}_A,(\widehat{t}_\lambda)_{\lambda \in \Lambda} \right)$ hold, where $\widehat{w}_B$ is defined in Proposition~\ref{l2_bound} and $\widehat{w}_A$ in Equation~\eqref{w2_def}.
Moreover, $x \mapsto \frac{\widehat{w}_A(x)}{x}$ is nonincreasing.
Therefore, Theorem~\ref{agcv_mean} applies with $\widehat{w}_{1,1} = \widehat{w}_A, \widehat{w}_{1,2} = \widehat{w}_A, \widehat{w}_{2,1} = \widehat{w}_B, \widehat{w}_{2,2} = \widehat{w}_A$, $x = \log n_v$ and it remains to bound the remainder terms $(R_{2,i})_{1 \leq i \leq 4}$ of Equation~\eqref{eq_2}. For each $i$, we bound $R_{2,i}(\theta)$ by an absolute constant times $\max \{T_1(\theta),T_2(\theta),T_3(\theta) \}$, where
\begin{align*}
 T_1(\theta) &= \frac{6\cvxct}{100} \frac{\log(n_v \lvert \Lambda \rvert)}{\theta n_v} \\
 T_2(\theta) &= \left( \slope \vee L \right)^2 \kappa C \frac{\log^2(n_v \lvert \Lambda \rvert)}{\theta^3 \lambda_m n_v^2} \\
 T_3(\theta) &= L(\slope \vee L) \kappa \frac{\log^{\frac{3}{2}}(n_v \lvert \Lambda \rvert)}{\theta \lambda_m n_v \sqrt{n_t}}\enspace .
\end{align*}
Summing up these bounds yields Theorem~\ref{rkhs_thm}. 

\subsubsection{Bound on $R_{2,1}(\theta) = \sqrt{2} \theta \mathbb{E} \left[ \delta^2\left( \widehat{w}_B, \frac{\theta}{2} \sqrt{\frac{n_v}{\log (n_v\lvert \Lambda \rvert)}}\right) \right]$}
Recall that
$ \widehat{w}_B(x)^2 := \max \left\{ \cvxct x^2, \  \slope\frac{4}{3}\sqrt{\frac{\kappa C}{\lambda_m}} x^3 + 10 \slope L \frac{\kappa \sqrt{\widehat{y}}}{\lambda_m \sqrt{n_t}} x^2   \right\}$.

By Equation~\eqref{rkhs_fp} in Lemma~\ref{fixed_pt} with $a = \sqrt{\cvxct}$,
$b = \slope \frac{4}{3} \sqrt{\frac{\kappa C}{\lambda_m}}$, $c = 10 \slope L \frac{\kappa \sqrt{\widehat{y}}}{\lambda_m \sqrt{n_t}}$,
\begin{equation} \label{eq_delta_2_1}
 \delta^2 \left( \widehat{w}_B, \frac{\theta}{2} \sqrt{\frac{n_v}{\log(n_v \lvert \Lambda \rvert)}} \right) \leq  4\cvxct \frac{\log(n_v \lvert \Lambda \rvert)}{\theta^2 n_v} + \ 29 \slope^2 \kappa C \frac{ \bigl[ \log(n_v \lvert \Lambda \rvert) \bigr]^2}{\theta^4 \lambda_m n_v^2} + 80 \slope L \kappa \frac{\bigl[ \log(n_v \lvert \Lambda \rvert) \bigr]\sqrt{\widehat{y}}}{\theta^2 \lambda_m n_v \sqrt{n_t}} \enspace .
\end{equation}
Therefore,
\[
 R_{2,1}(\theta) \leq 4\sqrt{2} \cvxct \frac{\log(n_v \lvert \Lambda \rvert)}{\theta n_v} + 29\sqrt{2} \slope^2 \kappa C \frac{ \bigl[ \log(n_v \lvert \Lambda \rvert) \bigr]^2}{\theta^3 \lambda_m n_v^2} + 80\sqrt{2} \slope L \kappa \frac{\bigl[ \log(n_v \lvert \Lambda \rvert) \bigr]\sqrt{\mathbb{E}[\widehat{y}]}}{\theta \lambda_m n_v \sqrt{n_t}} \enspace .
\]
By Proposition~\ref{y_bound}, $\mathbb{E}[\widehat{y}] \leq 4 + \log \lvert \Lambda \rvert$. Since $n_v \geq 100 \geq \mathrm{e}^4$, $\mathbb{E}[\widehat{y}] \leq \log(n_v\lvert \Lambda \rvert)$. As a result,
\begin{align*}
 R_{2,1}(\theta) &\leq 6 \cvxct \frac{\log(n_v \lvert \Lambda \rvert)}{\theta n_v} + 42 \slope^2 \kappa C \frac{ \bigl[ \log(n_v \lvert \Lambda \rvert) \bigr]^2}{\theta^3 \lambda_m n_v^2} + 114\slope L \kappa \frac{\bigl[ \log(n_v \lvert \Lambda \rvert) \bigr]^{\frac{3}{2}}}{\theta \lambda_m n_v \sqrt{n_t}} \\
 &\leq 100 T_1(\theta) + 42 T_2(\theta) + 114 T_3(\theta) \\
 &\leq 256 \times \max \left\{ T_1(\theta),T_2(\theta),T_3(\theta) \right\} \enspace .
\end{align*}

\subsubsection{Bound on $R_{2,2}(\theta) =  \frac{\theta^2}{2} \mathbb{E} \left[ \delta^2 \left(\widehat{w}_A,\frac{\theta^2}{4} \frac{n_v}{\log(n_v \lvert \Lambda \rvert)} \right) \right]$} 
Recall that by definition, $\widehat{w}_A(x) = L \sqrt{\frac{\kappa C}{\lambda_m}} x + 5L^2 \frac{\kappa \sqrt{\widehat{y}}}{\lambda_m \sqrt{n_t}}$ (Equation~\eqref{w2_def}).
By Equation~\eqref{lin_fp} in Lemma~\ref{fixed_pt} with $b = L \sqrt{\frac{\kappa C}{\lambda_m}}$ and $c = 5L^2 \frac{\kappa \sqrt{\widehat{y}}}{\lambda_m \sqrt{n_t}}$, we have 
\begin{align}
    \delta^2 \left(\widehat{w}_A,\frac{\theta^2}{4} \frac{n_v}{\log(n_v \lvert \Lambda \rvert)} \right) &\leq 16 L^2 \kappa C  \frac{ \log^2(n_v \lvert \Lambda \rvert)}{\theta^4 \lambda_m n_v^2} + 40 L^2 \kappa \frac{\bigl[ \log(n_v \lvert \Lambda \rvert) \bigr] \sqrt{\widehat{y}}}{\theta^2 \lambda_m n_v \sqrt{n_t}} \enspace . \label{eq_delta_2_2}
\end{align}
As $\mathbb{E}[\widehat{y}] \leq \log(n_v \lvert \Lambda \rvert)$ by Proposition~\ref{y_bound}, it follows that
\begin{align*}
 R_{2,2}(\theta) &\leq 8 L^2 \kappa C \frac{ \log^2(n_v \lvert \Lambda \rvert)}{\theta^2 \lambda_m n_v^2} + 20L^2 \kappa  \frac{\log^\frac{3}{2}(n_v \lvert \Lambda \rvert)}{ \lambda_m n_v \sqrt{n_t}} \\
 &\leq 8 \theta T_2(\theta) + 20 \theta T_3(\theta) \\
 &\leq 28 \times \max \left\{ T_1(\theta), T_2(\theta), T_3(\theta) \right\} \text{ since } \theta \in (0;1] \enspace .
\end{align*}

\subsubsection{Bound on $R_{2,3}(\theta) = \frac{1}{n_v} \left( \theta + \frac{2\bigl[1 + \log(\lvert \Lambda \rvert) \bigr]}{\theta} \right) \mathbb{E} \left[ \widehat{\delta}^2 \bigl(\widehat{w}_{A}, \sqrt{n_v} \bigr) \right]$}
By Equation~\eqref{lin_fp} in Lemma~\ref{fixed_pt} with $b = L \sqrt{\frac{\kappa C}{\lambda_m}}, c = 5L^2 \frac{\kappa \sqrt{\widehat{y}}}{\lambda_m \sqrt{n_t}}$, 
\begin{equation} \label{eq_delta_1_1}
 \delta^2(\widehat{w}_A,\sqrt{n_v}) \leq L^2 \frac{\kappa C}{\lambda_m n_v} + L^2 \frac{10\kappa \sqrt{\widehat{y}}}{\lambda_m \sqrt{n_v n_t}} \enspace . 
\end{equation}

As $\theta \in (0;1]$ and $n_v \geq 100 \geq \mathrm{e}^{\frac{3}{2}}$, we have $\theta + \frac{2}{\theta} \leq \frac{3}{\theta} \leq \frac{2\log n_v}{\theta}$, hence
\begin{equation}
 \theta + \frac{2(1 + \log(\lvert \Lambda \rvert))}{\theta} \leq \frac{2\log(n_v \lvert \Lambda \rvert)}{\theta} \enspace . \label{eq_simplif1}
\end{equation}
Therefore,
\[ R_{2,3}(\theta) \leq \frac{2\log(n_v\lvert \Lambda \rvert)}{\theta n_v} \left[ L^2 \frac{\kappa C}{\lambda_m n_v} + L^2 \frac{10\kappa \sqrt{\mathbb{E}[\widehat{y}]}}{\lambda_m \sqrt{n_v n_t}} \right] \enspace . \]
Since $\mathbb{E}[\widehat{y}] \leq \log(n_v \lvert \Lambda \rvert)$ by Proposition~\ref{y_bound},
\begin{align*}
 R_{2,3}(\theta) &\leq 2\log(n_v \lvert \Lambda \rvert)\frac{L^2 \kappa C }{\theta \lambda_m n_v^2} + 20 L^2 \kappa \frac{ \log^{\frac{3}{2}}(n_v \lvert \Lambda \rvert)}{\theta \lambda_m n_v \sqrt{n_v n_t}} \\
 &\leq \frac{2\theta^2}{\log(n_v \lvert \Lambda \rvert)} T_2(\theta) + \frac{20}{\sqrt{n_v}} T_3(\theta) \\
 &\leq 0.4 T_2(\theta) + 2 T_3(\theta) \text{ since } n_v \geq 100 \text{ and } \lvert \Lambda \rvert \geq 2\\
 &\leq 2.4 \times \max \{T_1,T_2,T_3 \} \enspace .
\end{align*}

\subsubsection{Bound on $R_{2,4}(\theta) = \frac{1}{n_v} \left( \theta + \frac{2\bigl[1 + \log(\lvert \Lambda \rvert) \bigr] + \log^2(\lvert \Lambda \rvert)}{\theta} \right) \mathbb{E} \left[ \widehat{\delta}^2 \bigl(\widehat{w}_{A}, n_v \bigr) \right]$}
By Equation~\eqref{lin_fp} in Lemma~\ref{fixed_pt} with $b = L \sqrt{\frac{\kappa C}{\lambda_m}}, c = 5L^2 \frac{\kappa \sqrt{\widehat{y}}}{\lambda_m \sqrt{n_t}}$, 
\begin{equation} \label{eq_delta_1_2}
 \delta^2(\widehat{w}_A,n_v) \leq L^2 \frac{\kappa C}{\lambda_m n_v^2} + L^2 \frac{10\kappa \sqrt{\widehat{y}}}{\lambda_m n_v \sqrt{n_t}} \enspace . 
\end{equation}
Since $\theta \in [0;1]$, $n_v \geq 100$ and $\lvert \Lambda \rvert \geq 2$, we have $\log(n_v \lvert \Lambda \rvert) \geq \log(200) \geq 5$ and
\begin{align*}
 \theta + \frac{2\bigl[1 + \log(\lvert \Lambda \rvert) \bigr]}{\theta}
 &\leq \frac{2\log(n_v \lvert \Lambda \rvert)}{\theta} \text{ by equation } \eqref{eq_simplif1} \\
 &\leq \frac{2\log^2(n_v \lvert \Lambda \rvert)}{5\theta} \enspace .
\end{align*}
Hence, by Equation~\eqref{eq_delta_1_2},
\[ R_{2,4}(\theta) \leq \frac{1,4 \log^2(n_v \lvert \Lambda \rvert)}{\theta n_v} \left[ L^2 \frac{\kappa C}{\lambda_m n_v^2} + L^2 \frac{10\kappa \sqrt{\mathbb{E}[\widehat{y}]}}{\lambda_m n_v \sqrt{n_t}} \right] \enspace . \]
Since $\mathbb{E}[\widehat{y}] \leq \log(n_v \lvert \Lambda \rvert)$,
\begin{align*}
 R_{2,4}(\theta) &\leq  1,4 \log^2(n_v \lvert \Lambda \rvert) \frac{L^2 \kappa C}{\theta \lambda_m n_v^3} + 14 L^2 \kappa \frac{\log^{\frac{5}{2}}(n_v \lvert \Lambda \rvert)}{\theta \lambda_m n_v^2 \sqrt{n_t}} \\
 &\leq  \frac{1,4\theta^2}{n_v} T_2(\theta) + 14 \frac{\log(n_v \lvert \Lambda \rvert)}{n_v} T_3(\theta) \enspace .
\end{align*}
Since $n_v \geq 100$ and $\lvert \Lambda \rvert \leq \mathrm{e}^{\sqrt{n_v}}$, we have $\frac{\log(n_v \lvert \Lambda \rvert)}{n_v} \leq \frac{\log(n_v)}{n_v} + \frac{\log(\mathrm{e}^{\sqrt{n_v}})}{n_v} \leq \frac{\log(100)}{100} + \frac{1}{10} \leq 0.15$ and so
\begin{align*}
 R_{2,4}(\theta) &\leq 0.014 T_2(\theta) + 2.1 T_3(\theta) \\
 &\leq 2.2 \times \max \{T_1(\theta),T_2(\theta),T_3(\theta) \} \enspace  .
\end{align*}

\subsubsection{Conclusion}
Summing up the above inequalities, we get that for every $\theta \in (0;1]$,
\begin{align*}
 R_2(\theta) &= R_{2,1}(\theta) + R_{2,2}(\theta) +  R_{2,3}(\theta) + R_{2,4}(\theta) \\
 &\leq 289 \max \{T_1(\theta),T_2(\theta),T_3(\theta) \} \enspace .
\end{align*}
Equation~\eqref{eq_2} in Theorem~\ref{agcv_mean} thus yields
\begin{align*}
  \Bigl(1 - \theta \Bigr)\mathbb{E}[\loss{\ERMag{\cT}}] 
 &\leq \Bigl( 1+\theta \Bigr)\mathbb{E}\Bigl[\min_{\lambda \in \Lambda} \loss{\learnrule_\lambda(D_{n_t})} \Bigr] +  289 \max \{T_1(\theta),T_2(\theta),T_3(\theta) \}
\end{align*}
which proves Theorem~\ref{rkhs_thm} with
$b_1 = 289 (\slope \vee L)^2 \kappa C$ and $b_2 = 289 L(\slope \vee L)\kappa$.
\qed

\section{Proof of Proposition~\ref{prop_sc} and Corollary~\ref{eps_reg}} \label{sec.proof_sc_cor} 

Let us start by two useful lemmas. 
\begin{lemma} \label{risk_deriv}
 If $\psi$ is a convex, Lipschitz-continuous, and even function, and $Y$ is a random variable with a non-atomic distribution, the function
 \[R: u \mapsto \mathbb{E}\bigl[ \psi(u-Y) \bigr]\] is convex and differentiable with derivative $R'(u) = \mathbb{E}[\psi'(u-Y)]$. 
 Moreover, if $Y$ is symmetric around $q$, i.e $(q-Y) \sim (Y-q)$, then $R$ reaches a minimum at~$q$.
\end{lemma}

\begin{proof}
First, remark that $R$ is convex by convexity of $\psi$. 
Let $u \in \mathbb{R}$.
 For $h \neq 0$, let $k(h,Y) = \frac{\psi(u+h-Y) - \psi(u-Y)}{h}$.
 Let $A$ be the set on which $\psi$ is non-differentiable. Since $\psi$ is convex, $A$ is at most countable.
 By definition, $k(h,Y) \underset{h \to 0}{\longrightarrow} \psi'(u-Y) $ whenever $u-Y \notin A$, that is to say $Y \notin u - A$. Since $Y$ is non-atomic, $\mathbb{P}(Y \notin u - A) = 1$. 
 Moreover, since $\psi$ is Lipschitz, there exists a constant $L$ such that $\forall h \neq 0, |k(h,Y)| \leq L$.
 Therefore, by the dominated convergence theorem,
 \[ \frac{R(u+h) - R(u)}{h} = \mathbb{E} [k(h,Y)] \underset{h \to 0}{\longrightarrow} \mathbb{E} [\psi'(u-Y)] \enspace . \]
 Thus, $R$ is differentiable and for all $u \in \mathbb{R}$, $R'(u) = \mathbb{E} [\psi'(u-Y)].$
 
Moreover, we have
 \begin{align*}
  R'(q) &= \mathbb{E}[\psi'(q-Y)] \\
  &= - \mathbb{E}[\psi'(Y-q)] \text{ since } \psi'(-x) = - \psi'(x) \text{ on } \mathbb{R} \backslash A \\
  &= - \mathbb{E}[\psi'(q-Y)] \text{ since } (Y-q) \sim (q-Y) \enspace , 
 \end{align*}
which implies that $R'(q) = 0$. Hence,
$R$ reaches a minimum at $q$ since $R$ is convex.
\end{proof}

\begin{lemma} \label{lem_min_cvx}
 Let $g: \mathbb{R} \rightarrow \mathbb{R}$ be a differentiable convex function that reaches a minimum at $u_* \in \mathbb{R}$. 
 If there exists $\varepsilon, \delta$ such that
 \begin{equation} \label{hyp_loc_min_grad}
   \forall u \in [u_* - \delta; u_* + \delta], \qquad |g'(u)| \geq \varepsilon \lvert u - u_* \rvert \enspace ,
 \end{equation}
 then for all $(u,v) \in \mathbb{R}^2$,
 \[ (u-v)^2 \leq \left[ \frac{4}{\varepsilon} \vee \left( \frac{4}{\varepsilon \delta} \lvert u-v \rvert \right)  \right] \bigl[ g(u) + g(v) - 2g(u_*) \bigr] \enspace . \]
\end{lemma}
\begin{proof}
By integrating Equation~\eqref{hyp_loc_min_grad},
\begin{equation} \label{loc_min_g}
 \forall u \in [u_* - \delta; u_* + \delta], (g(u) - g(u_*)) \geq \frac{\varepsilon}{2} (u - u_*)^2 \enspace .
\end{equation}
Let 
\begin{equation}
 h(u) = \frac{1}{\delta} \left[g(u_* + \delta) - g(u_*) \right] [u - u_* ] \enspace . \label{def_h_cord}
\end{equation}
By convexity of $g$, for any $u \geq u_* + \delta$, $g(u) - g(u_*) \geq h(u)$. Hence by Equation~\eqref{loc_min_g} with $u = u_* + \delta$ and Equation~\eqref{def_h_cord},
\begin{equation} 
 \forall u \geq u_* + \delta, g(u) - g(u_*) \geq \frac{1}{\delta} \frac{\varepsilon}{2} \delta^2 [u - u_*] = \frac{\varepsilon \delta}{2} [u - u_*] \enspace .
\end{equation}
The same argument applies to the convex function $g(- \cdot)$ with minimum $-u_*$, which yields
\begin{equation}
 \label{min_g_at_infty}
 \forall u \in \mathbb{R}, |u - u_*| \geq \delta \implies g(u) - g(u_*) \geq \frac{\varepsilon \delta}{2} |u - u_*| \enspace .
\end{equation}
Let $(u,v) \in \mathbb{R}^2$.
Assume without loss of generality that $|u - u_*| \geq |v - u_*|$.
If $|u - u_*| \leq \delta$ then by Equation~\eqref{loc_min_g},
\begin{align*}
(u-v)^2 &\leq 2\bigl[u-u_* \bigr]^2 + 2 \bigl[v-u_* \bigr]^2 \\ 
&\leq \frac{4}{\varepsilon} \bigl[ g(u) + g(v) - 2 g\bigl(u_* \bigr)\bigr] 
\enspace . 
\numberthis \label{eq_varl1_ub1}.
\end{align*}
Otherwise, by Equation~\eqref{min_g_at_infty},
\begin{align*}
(u-v)^2 &\leq |u-v|\bigl[|u - u_*| + |v - u_*| \bigr] \\
&\leq 2|u-v| |u-u_*| \\
&\leq \frac{4}{\varepsilon \delta} |u-v| \bigl[g(u) - g\bigl(u_* \bigr) \bigr] \\
&\leq \frac{4}{\varepsilon \delta} |u-v| \bigl[ g(u) + g(v) - 2 g(u_*)\bigr]
\enspace . 
\numberthis \label{eq_varl1_ub2}
\end{align*}
\end{proof}

\subsection{Proof of Proposition~\ref{prop_sc}} \label{sec.proof_sc}
Now, we can prove Proposition~\ref{prop_sc}.
Let $R_x: u \mapsto \int |u-y| dF_x(y)$. By Lemma~\ref{risk_deriv} with $\psi = |\cdot|$, for all $v \in \mathbb{R}$,
 \begin{align*}
  R_x'(v) &= \int \left[ - \mathbb{I}_{v - y \leq 0} + \mathbb{I}_{v-y \geq 0} \right] dF_x(y) \\
  &= F_x(v) - \bigl[1 - F_x(v) \bigr] \\
  &= 2 \left[ F_x(v) - F_x(\bayes(x)) \right]
 \end{align*}
since by definition, $F_x(\bayes(x)) = \frac{1}{2}$.
Hence by hypothesis \eqref{hyp_loc_incr_F},
for all $u \in [\bayes(x) - b(x); \bayes(x) + b(x)]$,
\[ |R_x'(u)| \geq 2a(x) |u - \bayes(x)|. \]
Therefore by Lemma~\ref{lem_min_cvx}, for all $x \in \cX$ and $(u,v) \in \mathbb{R}^2$,
\begin{align*}
 (u-v)^2 &\leq \left( \frac{4}{a(x)} \vee \frac{4 |u-v|}{a(x) b(x)} \right) \left[ R_x(u) + R_x(v) - 2R_x \bigl( \bayes(x) \bigr)  \right] \\
 &\leq \left( \frac{4}{a_m} \vee \left( \frac{4}{\mu_m} |u-v| \right) \right) \left[ R_x(u) + R_x(v) - 2R_x \bigl( \bayes(x) \bigr)  \right].
\end{align*}
Since $\cpred: (u,y) \mapsto |u-y|$, it follows by taking $x = X$ that
\[ \left( \cpred(u,Y) - \cpred(v,Y) \right)^2 \leq (u-v)^2 \leq \left( \frac{4}{a_m} \vee \left( \frac{4}{\mu_m} |u-v| \right) \right) \bigl[\xloss{u} + \xloss{v} \bigr] ,\]
which implies hypothesis $SC_{\frac{4}{a_m}, \frac{4}{\mu_m}}$.
\qed

\subsection{Proof of Corollary~\ref{eps_reg}} \label{sec.proof_cor_eps_reg}
Corollary~\ref{eps_reg} is a consequence of Theorem~\ref{rkhs_thm}. Let us check that its assumptions are satisfied. 

\paragraph{Compatibility hypothesis $(Comp_1(c^{eps}_0,c^{eps}_{\varepsilon}))$}
Fix $x \in \cX$ and let $p_x,F_x$ be the pdf and cdf corresponding to the distribution $Y$ given $X = x$. By assumption, $p_x$ is symmetric; $\bayes(x)$ can be chosen equal to the center of symmetry (recall that the contrast function here is $\cfun(t,(x,y)) = c_0^{eps}(t(x),y) = |t(x)-y|$, so any conditional median is a possible value for $\bayes(x)$). 
Let 
\begin{equation}
R_{\varepsilon, x}: u \mapsto \int_y c^{eps}_\varepsilon(u,y) p_x(y) \mathrm{d}y = \int \psi_{\varepsilon}(u-y) p_x(y) \mathrm{d}y \enspace , \label{eq_R_eps}
\end{equation}
where $\psi_{\varepsilon}(z) = (|z| - \varepsilon)_+ $ for any $z \in \mathbb{R}$. 
Lemma C.1 applies, since
$p_x$ is symmetric by assumption and $\psi_{\varepsilon}$ is even, convex and 1-Lipschitz. 

 Hence for any $\varepsilon \geq 0$, $R_{\varepsilon, x}$ has a minimum at $\bayes(x)$ and is differentiable, with
 \begin{align*}
   R'_{\varepsilon, x}(u) &= \int \psi_{\varepsilon}'(u-y) p_x(y) \mathrm{d}y = \int \left[-\mathbb{I}_{u-y \leq -\varepsilon} + \mathbb{I}_{u-y \geq \varepsilon} \right] p_x(y) \mathrm{d}y \\
   &= F_x(u-\varepsilon) - \left[ 1 - F_x(u + \varepsilon) \right] \enspace .
   \numberthis \label{risk_grad}
 \end{align*}
Therefore, for any $\varepsilon \geq 0$ and $u \in \mathbb{R}$,
\begin{equation} \label{eps_deriv_risk}
R'_{\varepsilon,x}(u) - R'_{0,x}(u) = \int_{0}^{\varepsilon} \left[ - p_x(u-t) + p_x(u+ t) \right] \mathrm{d}t \enspace .
\end{equation}
Now, assume 
that $u \geq \bayes(x)$. By symmetry of $p_x$ around $\bayes(x)$, for all $t \geq 0$,
\begin{align*}
p_x(u-t) &= p_x(\bayes(x) + (u - \bayes(x) - t)) \\
&= p_x(\bayes(x) + |u - \bayes(x) - t|) \enspace . \numberthis \label{unimod_min}
\end{align*}
Since $p_x$ is unimodal, its mode is $\bayes(x)$ and $p_x$ is non-increasing on $\left[\bayes(x); + \infty \right)$. It follows from Equation~\eqref{unimod_min} that for all $u \geq \bayes(x)$ and $t \geq 0$, 
\begin{align*}
p_x(u-t)
&\geq p_x(\bayes(x) + |u - \bayes(x)| + t) \\
&= p_x(u+t). \numberthis \label{px_ineq}
\end{align*}
Therefore, by Eq.~\eqref{eps_deriv_risk} and \eqref{px_ineq}, for all $u \geq \bayes(x)$ and $\varepsilon \geq 0$,
$R'_{\varepsilon,x}(u) \leq R'_{0,x}(u)$.
By integration, this implies that for all $u \geq \bayes(x)$,
\begin{equation} 
R_{\varepsilon,x}(u) - R_{\varepsilon,x}(\bayes(x)) \leq R_{0,x}(u) -  R_{0,x}(\bayes(x)) \enspace . \label{eq_comp_R}
\end{equation}
By Equation~\eqref{eq_R_eps} and symmetry of $p_x$, $R_{\varepsilon,x}$ and $R_{0,x}$ are symmetric around $\bayes(x)$, hence inequality \eqref{eq_comp_R} is also valid when $u \leq \bayes(x)$.
Taking $x = X$, $u = t(X)$ and integrating, we get
$\risk_{c_\varepsilon^{eps}}(t) - \risk_{c_\varepsilon^{eps}}(s) \leq \risk_{c_0^{eps}}(t) - \risk_{c_0^{eps}}(s)$ which proves $Comp_1(c^{eps}_0,c_\varepsilon^{eps})$.

\paragraph{Hypothesis $SC_{4\sigma,8}$}
We first compute a lower bound on $R_{0,x}$. 

Let $q_{x,\frac{1}{4}} = \sup \{ y | F_x(y) \leq \frac{1}{4} \}$ and $q_{x,\frac{3}{4}} = \inf \{ y | F_x(y) \geq \frac{3}{4} \}$. By continuity of $F_x$, $F_x(q_{x,\frac{1}{4}}) = \frac{1}{4}$ and $F_x(q_{x,\frac{3}{4}}) = \frac{3}{4}$. 
Let $\sigma(x) = q_{x,\frac{3}{4}} - q_{x,\frac{1}{4}}$, which is the smallest determination of the interquartile range. By symmetry of $p_x$ around $\bayes(x)$, $\frac{1}{2} \bigl[ q_{x,\frac{1}{4}} + q_{x,\frac{3}{4}} \bigr] = \bayes(x)$, therefore $q_{x,\frac{3}{4}} = \bayes(x) + \frac{\sigma(x)}{2}$ and $q_{x,\frac{1}{4}} = \bayes(x) - \frac{\sigma(x)}{2}$. 

For any $u \in \left[\bayes(x) - \frac{\sigma(x)}{2} ; \bayes(x) + \frac{\sigma(x)}{2} \right]$, by symmetry of $p_x$ around $\bayes(x)$,
\begin{align*}
\left| F_x(u) - F_x \bigl( \bayes(x) \bigr) \right| 
&= \int_{\bayes(x)}^{\bayes(x) + |u - \bayes(x)|} 2p_x(v) \mathrm{d}v 
\\
&= |u-\bayes(x)| \frac{1}{|u - \bayes(x)|} \int_{\bayes(x)}^{\bayes(x) + |u - \bayes(x)|} 2p_x(v) \mathrm{d}v \enspace .
\end{align*}
Since $p_x$ is non-increasing on $[\bayes(x); + \infty)$ and $|u - \bayes(x) | \leq \frac{\sigma(x)}{2}$,
\begin{align*}
\left| F_x(u) - F_x \bigl( \bayes(x) \bigr) \right| &\geq |u - \bayes(x)| \frac{2}{\sigma(x)} \int_{\bayes(x)}^{\bayes(x) + \frac{\sigma(x)}{2}} 2p_x(v) \mathrm{d}v  \\
&= |u - \bayes(x)| \frac{4}{\sigma(x)} \left[ F_x \bigl( q_{x,\frac{3}{4}} \bigr) - F_x \bigl(\bayes(x) \bigr) \right] \\
&= \frac{|u - \bayes(x)|}{\sigma(x)} \enspace .
\end{align*}
Hence, by Proposition~\ref{prop_sc} with $a(x) = \frac{1}{\sigma(x)}$ and $b(x) = \frac{\sigma(x)}{2}$, $(\cpred,X,Y)$ satisfies hypothesis $SC_{4\sigma,8}$.

\paragraph{Conclusion}
To conclude, we apply Theorem~\ref{rkhs_thm} with $\kappa = 1, C = 1, 
 L = 1$ (since  $c^{eps}_0$ and  $c^{eps}_\varepsilon$ are 1-Lipschitz), 
 $\cvxct = 4\sigma$ and $\slope = 8$.
 Since constants $b_1,b_2$ of Theorem~\ref{rkhs_thm} only depend on $\kappa,L,C,\slope$ and all these parameters have now received explicit values, the constants $b_1,b_2$ are now absolute. 

\section{Classification: proof of Theorem~\ref{thm_classif}} \label{sec_classif}
In the proof of Theorem~\ref{agcv_mean}, we used convexity of the risk to show that the risk of the average was less than the average of the risk. 
A property of this type also holds in the setting of classification, with the average replaced by the majority vote.

\begin{proposition}
\label{maj_vote}
In the classification classification ---see Example~\ref{classif_pbm}---,
let $(\ERM{i})_{1 \leq i \leq V}$ denote a finite family of functions $\cX \rightarrow \cY$ and let 
$\ERM{}^{\text{mv}}$ be some majority vote rule: 
$\forall x \in \cX$, $\ERM{}^{\text{mv}}(x) \in \argmax_{y \in \cY} \lvert \{i \in [V] : \ERM{i}(x) = m \} \rvert $. 
Then, 
\[ 
\ell(\bayes,\ERM{}^{\text{mv}}) 
\leq \frac{M}{V} \sum_{i = 1}^V \ell(\bayes,\ERM{i})
\qquad \text{and}\qquad 
\risk(\ERM{}^{\text{mv}}) \leq  \frac{2}{V} \sum_{i = 1}^V \risk(\ERM{i}) 
\enspace.\]
\end{proposition}

\begin{proof}
For any $y \in \cY$, define $\eta_y: x \mapsto \mathbb{P}[Y = y | X = x]$. 
Then, for any $f \in \parspace$, $\risk(f) = \bE [ 1 - \eta_{f(X)} (X) ]$ hence 
$\bayes(X) \in \argmax_{y \in \cY} \eta_y(X) $ 
and 
\[ 
\ell(\bayes,f) 
= \bE \Bigl[ \max_{y \in \cY} \eta_y(X) - \eta_{f(X)}(X) \Bigr] 
= \bE \bigl[ \eta_{\bayes(X)}(X) - \eta_{f(X)}(X) \bigr] 
\enspace. 
\]
We now fix some $x \in \cX$ and define $\mathcal{C}_x(y) = \{i \in [V]: \ERM{i}(x) = y \}$ 
and $C_x = \max_{y \in \cY} |\mathcal{C}_x(y)|$.
Since $C_x M\ge \sum_{y \in \cY} |\mathcal{C}_x(y)| = V $,
it holds $C_x \geq V/M$.
On the other hand, by definition of $\ERM{}^{\text{mv}}$,
\begin{align*}
 \frac{1}{V} \sum_{i = 1}^V \bigl[ \underbrace{ \eta_{\bayes(x)}(x) - \eta_{\ERM{i}(x)}(x) }_{\geq 0} \bigr]
 &\geq \frac{C_x}{V} \bigl( \eta_{\bayes(x)}(x) - \eta_{\ERM{}^{\text{mv}}(x)}(x) \bigr) 
 \geq \frac{1}{M} \bigl( \eta_{\bayes(x)}(x) - \eta_{\ERM{}^{\text{mv}}(x)}(x) \bigr) 
\enspace.
\end{align*}
Integrating over $x$ (with respect to the distribution of $X$) 
yields the first bound.

For the second bound, 
fix $x \in \cX$ and define $\mathcal{C}_x(y)$ and $C_x$ as above. 
Let $y \in \cY$ be such that $\ERM{}^{\text{mv}}(x) \neq y$. 
Since $y$ occurs less often than 
$\ERM{}^{\text{mv}}(x)$ 
among $\ERM{1}(x) , \ldots, \ERM{V}(x)$, 
we have $|\mathcal{C}_x(y)| \leq V/2$. 
Therefore,
\[
\frac{1}{V} \sum_{i = 1}^V \mathbb{I}_{\{ \ERM{i}(x) \neq y \} }
= \frac{V - |\mathcal{C}_x(y)|}{V} 
\geq \frac{1}{2} 
\enspace.\]
Thus
\[ 
\ERM{}^{\text{mv}}(x) \neq y 
\implies  
\frac{1}{V} \sum_{i = 1}^V \mathbb{I}_{\{ \ERM{i}(x) \neq y \} }
\geq \frac{1}{2}
\enspace.\]
Hence, for any $y \in \cY$, 
\[ 
\mathbb{I}_{\{\ERM{}^{\text{mv}}(x) \neq y \} }
\leq \frac{2}{V} 
\sum_{i = 1}^V \mathbb{I}_{\{ \ERM{i}(x) \neq y \} }
\enspace.\]
Taking expectations with respect to $(x,y)$ yields $\risk(\ERM{}^{\text{mv}} ) \leq 2 V^{-1} \sum_{i = 1}^V \risk(\ERM{i})$.
\end{proof}

We can now proceed with the proof of Theorem~\ref{thm_classif}.
\
\begin{proof}
The proof relies on a result by \cite[Eq.~(8.60), which is itself a consequence of Corollary~8.8]{Mas:2003:St-Flour}, which holds true as soon as 
\begin{equation} \label{eq.hyp-Eq-8.60-Massart}
\forall t \in \parspace, \qquad 
\text{Var} \bigl( \mathbb{I}_{\{ t(X) \neq Y \}} - \mathbb{I}_{ \{\bayes(X) \neq Y\} } \bigr) 
\leq  \Bigl[ w \bigl(\sqrt{\ell(\bayes,t)} \bigr) \Bigr]^2 
\end{equation}
for some nonnegative and nondecreasing continuous function $w$ on $\mathbb{R}^+$, 
such that $x \mapsto w(x)/x$ is nonincreasing on $(0,+\infty)$ and $w(1) \geq 1$. 

Let us first prove that assumption \eqref{eq.hyp-Eq-8.60-Massart} holds true. 
On one hand, since $\cY = \{0,1\}$, for any $t \in \parspace$, 
\begin{multline}
\label{eq.maj-Var-incr}
\text{Var} \bigl(\mathbb{I}_{\{  t(X) \neq Y\} } - \mathbb{I}_{ \{\bayes(X) \neq Y\} } \bigr)
\leq
\bE[|  \mathbb{I}_{ \{t(X) \neq Y\} } - \mathbb{I}_{ \{\bayes(X) \neq Y \}} |^2] \\
= \bE[ \mathbb{I}_{ \{t(X) \neq \bayes(X) \}} ] 
= \bE \bigl[ \lvert t(X) - \bayes(X)|]
\enspace.
\end{multline}
On the other hand, since we consider binary classification with the 0--1 loss, 
for any $t \in \parspace$ and $h>0$, 
\begin{align*} 
\ell(\bayes, t) 
&= \bE \bigl[ |2\eta(X) - 1| \cdot |t(X) - \bayes(X)| \bigr] 
\qquad  &\text{by \cite[Theorem~2.2]{Dev_Gyo_Lug:1996}} 
\\
&\geq h \bE \bigl[ |t(X) - \bayes(X)| \mathbb{I}_{\{|2\eta(X) - 1| \geq h\}} \bigr] 
\qquad  &
\\
&\geq h \bE \bigl[ |t(X) - \bayes(X)|  - \mathbb{I}_{\{|2\eta(X) - 1| < h\}} \bigr] 
\qquad  &\text{since } \NormInfinity{t - \bayes} \leq 1
\\
&\geq h \bE \bigl[ |t(X) - \bayes(X)| \bigr] - r h^{\beta+1}
\qquad  &\text{by \eqref{hyp.MA}.} 
\end{align*}
This lower bound is maximized by 
taking 
\[ 
h = h_* := \left( \frac{ \bE \bigl[ |t(X) - \bayes(X)| \bigr] } { r(\beta+1) } \right)^{\frac{1}{\beta}} 
\enspace , 
\]
which belongs to $[0,1]$ since 
$r \geq 1$  
and $\bE \bigl[ |t(X) - \bayes(X)| \bigr] \leq 1$. 
Thus, we obtain  
\begin{align*} 
\ell(\bayes, t) 
\geq 
h_* \frac{\beta}{\beta+1} \bE \bigl[ |t(X) - \bayes(X)| \bigr]
= \frac{\beta}{(\beta+1)^{(\beta+1)/\beta}  r^{1/\beta}} \bE \bigl[ |t(X) - \bayes(X)| \bigr]^{(\beta+1)/\beta} 
\end{align*}
hence Eq.~\eqref{eq.maj-Var-incr} leads to 
\begin{align*}
\text{Var} \bigl( \mathbb{I}_{ \{t(X) \neq Y \}} - \mathbb{I}_{ \{\bayes(X) \neq Y\} } \bigr)
\leq
\bE \bigl[ |t(X) - \bayes(X)| \bigr] 
&\leq 
\frac{\beta+1}{\beta^{\beta/(\beta+1)}} r ^{\frac{1}{\beta + 1}} \ell(\bayes,t)^{\frac{\beta}{\beta + 1}} 
\leq 2r ^{\frac{1}{\beta + 1}}\ell(\bayes,t)^{\frac{\beta}{\beta + 1}} 
\enspace.
\end{align*}
Therefore, Eq.~\eqref{eq.hyp-Eq-8.60-Massart} holds true with 
$w(u) = \sqrt{r_1} u^\frac{\beta}{\beta + 1}$ and $r_1=2r ^{\frac{1}{\beta + 1}}$, 
which statisfies the required conditions.
So, by \cite[Eq.~(8.60)]{Mas:2003:St-Flour}, for any $\theta \in (0,1)$, 
\begin{equation} \label{eq:ho_oracle}
\bE \bigl[ \lossb{\ERMho{ T}} \,\vert\, D_n^T \bigr] 
\leq \frac{1+\theta}{1-\theta} \inf_{m \in \cM} \lossb{\learnrule_m(D_{n}^T)} + \frac{\delta_*^2}{1-\theta} 
\left[ 2\theta + \log(e \lvert \cM \rvert)\left(\frac{1}{3} + \theta^{-1} \right) \right] 
\end{equation} 
where $\delta_*$ is the positive solution of the fixed-point equation 
$w(\delta_*) = \sqrt{n_v} \delta_*^2$, that is 
$\delta_*^2 = ( r_1 / n_v )^\frac{\beta + 1}{\beta + 2} 
$. 
Taking expectations with respect to the training data $D_n^T$, 
we obtain 
\begin{align*}
\bE \bigl[ \loss{\ERMho{T}} \bigr] 
&\leq 
\frac{1+\theta}{1 - \theta} \bE\left[ \inf_{m \in \cM} \lossb{\learnrule_m(D_{n}^T)} \right] +\frac{2 r^\frac{1}{\beta + 2}}{1-\theta} \frac{2\theta + \log(e \lvert \cM \rvert)\left(\frac{1}{3} + \theta^{-1} \right)}{n_v^\frac{\beta + 1}{\beta + 2} }
\enspace.
\end{align*}
Under assumptions \eqref{hyp.T},
$\bE \bigl[ \ell(\bayes,\ERMho{T}) \bigr] $ and $\bE \bigl[ \risk(\ERMho{T}) \bigr]$ 
do not depend on $T \in \cT$  
(they only depend on $T$ through its cardinality $n_t$). 

Now, by Proposition~\ref{maj_vote} applied to $(\ERMho{T})_{T \in \cT}$, 
\[ 
\bE \bigl[ \ell(\bayes,\ERMmv{\cT}) \bigr] 
\leq 2 \bE \bigl[ \ell(\bayes,\ERMho{ T_1}) \bigr] 
\leq 2\frac{1+\theta}{1 - \theta} \bE \left[ \inf_{m \in \cM} \lossb{\learnrule_m(D_{n}^T)} \right] 
  + \frac{4 r^\frac{1}{\beta + 2}}{1-\theta} \frac{2\theta + \log(e \lvert \cM \rvert)\left(\frac{1}{3} + \theta^{-1} \right)}{n_v^\frac{\beta + 1}{\beta + 2} } 
\enspace.
\]
Taking $\theta = 1/5$ leads to the result. 
\end{proof}

\bibliographystyle{plain}
\bibliography{premier_article_biblio}

\end{document}